\documentclass[11pt]{amsart}

\usepackage{amsrefs}
\usepackage{amsmath}
\usepackage{amssymb}
\usepackage{hyperref}
\usepackage{stackrel}
\usepackage{xcolor}
\usepackage{enumerate}
\usepackage{verbatim}
\usepackage{cancel}
\usepackage{setspace}
\allowdisplaybreaks

\newtheorem{teo}{Theorem}[section]
\newtheorem{lem}[teo]{Lemma}
\newtheorem{defi}[teo]{Definition}

\newtheorem{cor}[teo]{Corollary}
\newtheorem{pro}[teo]{Proposition}

\newtheorem{obs}[teo]{Remark}

\DeclareMathOperator{\supp}{supp}
\DeclareMathOperator{\dist}{dist}
\DeclareMathOperator{\Hess}{Hess}
\DeclareMathOperator{\spann}{span}

\newcommand{\zetaastv}{(\zeta_\varepsilon \ast_v F)}

\title{Mollifier smoothing of $C^0$-Finsler structures}

\author{Ryuichi Fukuoka}
\address{Department of Mathematics, State University of Ma\-rin\-g\'a,
87020-900, Ma\-rin\-g\'a, PR, Brazil \\ email: rfukuoka@uem.br}

\author{Anderson Macedo Setti}
\address{Department of Mathematics, Federal University of Rondon\'opolis,
78736-900, Rondon\'opolis, MT, Brazil \\ email: andersonsetti@ufmt.br}

\DeclareMathOperator{\Hessian}{Hess}



\date{}

\begin{document}

\begin{abstract}
A $C^0$-Finsler structure is a continuous function $F:TM \rightarrow [0,\infty)$ defined on the tangent bundle of a differentiable manifold $M$ such that its restriction to each tangent space is an asymmetric norm. 
We use the convolution of $F$ with the  standard mollifier in order to construct a mollifier smoothing of $F$, which is a one parameter family of Finsler structures $F_\varepsilon$ that converges uniformly to $F$ on compact subsets of $TM$ as $\varepsilon$ converges to zero. 
We prove that when $F$ is a Finsler structure, then the Chern connection, the Cartan connection, the Hashiguchi connection, the Berwald connection and the flag curvature of $F_\varepsilon$ converges uniformly on compact subsets to the corresponding objects of $F$.
As an application of this mollifier smoothing, we study examples of two-dimensional piecewise smooth Riemannian manifolds with nonzero total curvature on a line segment.
We also indicate how to extend this study to the correspondent piecewise smooth Finsler manifolds.
\end{abstract}

\maketitle

 
\section{Introduction}
\label{introducao}

Let $M$ be a differentiable manifold. Denote its tangent space at $x\in M$ by $T_xM$ and its cotangent space by $T^\ast_xM$. 
Let $TM = \{(x,y);x \in M, y\in T_xM\}$ be its tangent bundle and $T^\ast M = \{(x,\rho); x\in M, \rho \in T^\ast_xM\}$ be its cotangent bundle.
A {\em Finsler structure} on $M$ is a function $F:TM \rightarrow \mathbb R$ which is smooth on the {\em slit tangent bundle} $TM\backslash 0 := \{(x,y)\in TM;y\neq 0\}$ and such that its restriction to each tangent space $F(x,\cdot): T_xM \rightarrow \mathbb R$ is a Minkowski norm.
A differentiable manifold endowed with a Finsler structure is a {\em Finsler manifold}.
A reference book for this subject is \cite{BaoChernShen}. 

The development of Finsler geometry has followed the footsteps of Riemannian geometry, with the use of differential calculus in order to study geometric objects such as connections, curvature and geodesics.
On the other hand there are differences between Finsler manifolds and Riemannian manifolds.
For instance, non-Riemannian Finsler manifolds don't admit a connection which is symmetric and compatible with the metric. In order to overcome this shortcoming, we have several connections on Finsler manifolds (see \cite{BaoChernShen}). 
Moreover, non-Riemannian Finsler manifolds don't admit a canonical volume form (see \cite{Duran-volume}).

A {\em $C^0$-Finsler structure} on $M$ is a continuous function $F:TM\rightarrow \mathbb R$ such that $F(x, \cdot): T_xM \rightarrow \mathbb R$ is an asymmetric norm (An asymmetric norm is a norm without the symmetry condition $F(x,y)=F(x,-y)$).
They are generalizations of Finsler structures and they appear naturally among intrinsic invariant metrics on homogeneous spaces.
More precisely, Berestovski\u \i\ proved in \cite{Berestovskii2} that if $M$ is endowed with a locally compact, locally contractible, homogeneous intrinsic metric $d_M$ (distance function), then $(M,d_M)$ is isometric to the left coset manifold $G/H$ of a Lie group $G$ by a compact subgroup $H$ endowed with a $G$-invariant $C^0$-Carnot-Carath\'eodory-Finsler metric.
Moreover if every orbit of one-parameter subgroups of $G$ under the natural action $G\times G/H \rightarrow G/H$ is rectifiable, then $(M,d_M)$ is isometric to a $C^0$-Finsler manifold.

The theory of $C^0$-Finsler geometry is much less developed than the theory of Finsler geometry because differential calculus can't be applied directly on $C^0$-Finsler structures.
Their geodesics behave very differently from geodesics on Finsler manifolds.
For instance, $\mathbb R^2$ with the maximum norm is naturally isometric to a $C^0$-Finsler manifold and its geodesics aren't necessarily smooth. 
Moreover, if $x \in \mathbb R^2$ and $y \in T_x\mathbb R^2$, then there exist infinite geodesics $\gamma$ such that $\gamma(0)=x$ and $\gamma^\prime (0)=y$.
Another example can be found in \cite{Fukuoka-large-family}, where the author creates a large family of projectively equivalent $C^0$-Finsler manifolds such that ``most of'' their minimizing paths are concatenation of two line segments.
The elements of this family can be obtained by continuous deformations of a fixed $C^0$-Finsler manifold and their minimizing paths have a huge stability that doesn't happen in Finsler geometry.

In this work we construct a {\em mollifier smoothing} of $F$, which is a one parameter family of Finsler structures $F_\varepsilon:TM \rightarrow \mathbb R$ that converges uniformly to $F$ on compact subsets of $TM$ (see Theorem \ref{propriedades da aplicacao F epsilon definida em TM}). 
We prove that if $F$ is a Finsler structure, then the Chern connection, the Cartan connection, the Hashiguchi connection, the Berwald connection and the flag curvature of $F_\varepsilon$ converge uniformly to the respective objects of $F$ on compact subsets (see Theorems  \ref{convergencia uniforme Chern} and \ref{curvatura flag converge uniforme}). 
Therefore the mollifier smoothing works well in the smooth case and we can study geometrical objects of a $C^0$-Finsler manifold $(M,F)$ using approximation by $(M,F_\varepsilon)$. 

\begin{obs}
In \cite{CordovaFukuokaNeves} and \cite{Fukuoka-large-family}, the first author and his collaborators used the term $C^0$-Finsler structure for a function $F: TM \rightarrow \mathbb R$ such that $F(x, \cdot): T_xM \rightarrow \mathbb R$ is a norm.
In this work the term $C^0$-Finsler structure is modified in order to work with non-symmetric structures.
As it happens in Finsler geometry, terms like symmetric (or absolutely homogeneous) $C^0$-Finsler structure will be used in order to deal with the case where $F(x,\cdot): T_xM \rightarrow \mathbb R$ is a norm.
\end{obs}

\begin{obs}
In this work, the Einstein summation convention is in place.
\end{obs}

We outline the construction of $F_\varepsilon$. 
Let $(M,F)$ be a $C^0$-Finsler  manifold.
Let $(x^i) = (x^1,\ldots,x^n): U \rightarrow {\mathbb R^n}$ be a coordinate system in an open subset $U$ of $ M$.
As usual, we denote the basis on the fibers of  $TU$ and $T^*U$ by $\left\{ \frac{\partial}{\partial x^i}\right\} $ and $ \{ dx^i\} $ respectively. 
The coordinate system $(x^i)$ induces the {\em natural coordinate system} $((x^i),(y^i)) = (x^1, \dots, x^n, y^1, \ldots, y^n)$ on $TU$, where
\[
\left( x , y^i \frac{\partial}{\partial x^i} \right) \in TU  \cong  (x^1(x), \ldots, x^n(x), y^1, \ldots, y^n) \in \mathbb R^{2n}.
\]
The tangent bundle $TU$ is identified with $(x^i)(U) \times \mathbb R^n$ and a $C^0$-Finsler structure $F$ is denoted by $F(x,y) = F(x^1, \ldots, x^n, y^1, \ldots, y^n)$.

The mollifier smoothing of $F$ is done in two steps: 
A vertical smoothing, which takes place in each tangent space and a horizontal smoothing, which is done along $M$.

For the vertical smoothing, consider a coordinate system $(x^i):U_0 \subset M\rightarrow {\mathbb R^n}$ defined on an open subset $U_0$.
Let $U$ be an open subset of $M$ with compact closure such that  $\overline{U} \subset U_0$. 
Let $((x^i), (y^i))$ be the corresponding natural coordinate system on $TU$.
Let $\eta:\mathbb R^n \rightarrow \mathbb R$ be the standard mollifier, set $\eta_\varepsilon(y) = \frac{1}{\varepsilon^n}\eta(\frac{y}{\varepsilon})$ and define
\begin{equation}
\label{define zeta epsilon}
\zeta_\varepsilon = (1-u(\varepsilon))\eta_\varepsilon+ u(\varepsilon) \eta_{r},
\end{equation}
where $r>0$ will be defined afterwards and $u:(0,1) \rightarrow \mathbb (0,\infty)$ is a fixed but arbitrary strictly increasing function such that $\lim_{\varepsilon \rightarrow 0} u(\varepsilon) =0$.
Let
\begin{equation}
\zetaastv (x,y) := \int \zeta_\varepsilon(z)F(x,y-z)dz \label{suaviza vertical}
\end{equation}
be the convolution along the fibers of $TU$.
This function is smooth in each tangent space but $\zetaastv (x, \cdot)$ is not an asymmetric norm. 
However there exist $r_U>0$ such that for every $x \in U$,  $\zetaastv^{-1}(r_U) \cap T_xU$ is diffeomorphic to a smooth sphere of a Minkowski norm in $T_xU$.
The {\em local vertical smoothing of $F$ is} the continuous function
\[
G_\varepsilon:TU \rightarrow \mathbb R
\]
such that $G_\varepsilon(x, \cdot):T_xM \rightarrow \mathbb R$ is the Minkowski norm with sphere of radius $r_U$ equals to $ \zetaastv^{-1}(r_U) \cap T_xU$. 
The perturbation $u(\varepsilon) \eta_r$ of $(1-u(\varepsilon ))\eta_\varepsilon$ is used in order to assure that the sphere $\zetaastv^{-1}(r_U) \cap T_xU$ is strongly convex.
The only property lacking for $G_\varepsilon$ to be a Finsler structure is its horizontal differentiability, that is, the smoothness with respect to the variables $(x^1, \ldots, x^n)$. 
So we define the {\em local horizontal smoothing} $F_{\varepsilon,U_\varepsilon}$ by
\begin{equation}
F^2_{\varepsilon,U_\varepsilon}(x,y) =\int \eta_\varepsilon(x-z)G^2_\varepsilon(z,y) dz, \, (x,y) \in TU_\varepsilon, \label{suavizacao horizontal local}
\end{equation}
where $U_\varepsilon=\{x\in U| \dist_{\mathbb R^n}(x,\partial U)> \varepsilon\}.$
This function is a Finsler structure on $U_\varepsilon$.
Finally, using a differentiable partition of unity, we obtain a mollifier smoothing $F_\varepsilon:TM \rightarrow \mathbb R$ of $F$. 

In \cite{Setti}, the author consider
\[
\zeta_\varepsilon = (1-\varepsilon)\eta_\varepsilon + \varepsilon \eta_r
\]
instead of (\ref{define zeta epsilon}) and define the local horizontal smoothing by 
\begin{equation}
\label{suavizacao horizontal local 2}
F_{\varepsilon,U_\varepsilon}(x,y) =\int \eta_\varepsilon(x-z)G_\varepsilon(z,y) dz, \, (x,y) \in TU_\varepsilon{\color{red}.}
\end{equation}
He proves that all the results of Sections \ref{secao mollifier smoothing F}, \ref{secao smoothing finsler}
and \ref{secao convergencia conexoes curvatura} of the present work hold with (\ref{suavizacao horizontal local}) replaced by (\ref{suavizacao horizontal local 2}).
But we think that calculations with the smoothing (\ref{suavizacao horizontal local}) are easier to be done and more flexible for applications.
For instance, if $F$ is a Finsler structure, then the fundamental tensor of $F_{\varepsilon, U_\varepsilon}$ is the mollifier smoothing of the fundamental tensor of $F$.

The importance of mollifier smoothings in the theory of partial differential equations is well known. 
It is used to approximate locally integrable functions by smooth functions.
However its application in geometric structures is much less common.
In \cite{Davini}, the author studies the local theory of the weak Finsler structures. Let $U \subset {\mathbb R^n}$ be an open subset. A  Borel  measurable function $F:U \times {\mathbb R^n} \rightarrow \mathbb R $ is a weak Finsler structure in $U$ if $F(x,\cdot)$ is positively homogeneous for every $x\in \bar U$, $F(x,\cdot)$ is convex almost everywhere and there exist $a,b>0$ such that $a\Vert \xi \Vert \leq F(x, \xi) \leq b \Vert \xi\Vert$ for every $(x,\xi) \in \bar U \times \mathbb R^n$. 
In Theorem 4.5 of \cite{Davini}, the author considers the case where $F$ is continuous and he define a horizontal smoothing of $F$ essentially as (\ref{suavizacao horizontal local 2}). 

In \cite{Fukuoka-smoothing-riemannian}, the author considers Riemannian metrics $g$ with coefficients $g_{ij}$ in some local Sobolev space. 
The mollifier smoothing of these coefficients gives locally a one parameter family of Riemannian metrics $g_\varepsilon$ of class $\mathit{C}^\infty$ and it is equivalent to the horizontal smoothing (\ref{suavizacao horizontal local}) for the Sobolev setting.
The author shows that if $g$ is of class $\mathit{C}^\infty,$ then the Riemannian connection and the curvature tensor of $g_\varepsilon$ converges to the respective objects of $g$ as $\varepsilon$ converges to zero.
In this present work we prove some generalizations of this result for the Finsler setting (see Theorems \ref{convergencia uniforme Chern} and
\ref{curvatura flag converge uniforme}).

The main contribution of this work is to bring a tool from functional analysis to geometry, prove that it is effective when applied to the smooth case and provide an example of a non-smooth case where calculations can be done. 
This approximation of $C^0$-Finsler manifolds by Finsler manifolds is in the spirit of mollifier smoothings, where the derivatives of the smoothing converges to the correspondent weak derivatives of the smoothed function: 
In fact, the connections depend on $F^2_\varepsilon$ and its derivatives of order up to three, the flag curvature depends on $F^2_\varepsilon$ and its derivatives of order up to four and we expect that their limits give us geometrical information about $(M,F)$.
The main technical difficulty is the control of the vertical smoothing, which is done in spheres instead of in $F$ directly.

This work is organized as follows: In  Section \ref{preliminares} we present preliminary results that are necessary for the development of this work.
We present the theory of mollifier smoothings, asymmetric norms, Minkowski norms and Finsler geometry. Although lengthy, we think that this presentation is necessary because it deals with topics that aren't usually presented together.
In Section \ref{secao mollifier smoothing} we prove that the mollifier smoothing of a convex function is convex.
The strongly convex case is analysed because it is necessary that $F^2_\varepsilon$ has positive definite Hessian with respect to $(y^1, \ldots, y^n)$.
In Section \ref{section growth rates} we control the growth rate of asymmetric norms in an ``almost'' radial direction.
This is important in order to guarantee  the existence of a regular value $r_U > 0$ of  $\zetaastv$ such that the level set $\zetaastv^{-1}(r_U) \cap T_xU$ is the sphere of a Minkowski norm for every sufficiently small $\varepsilon$. 
Section \ref{section vertical smoothing} is devoted to the definition of the local vertical smoothing and the study of its properties.
In Section \ref{secao mollifier smoothing F} we define the mollifier smoothing $F_\varepsilon$ of $F$ and we prove that $F_\varepsilon$ converges uniformly to $F$ on compact subsets of $TM$ as $\varepsilon$ goes to zero.
In Section \ref{secao smoothing finsler}, we prove that if $F$ is a Finsler structure, then $G_\varepsilon$ and its partial derivatives converge uniformly to $F$ and their respective partial derivatives on compact subsets of $TU$.
Analogous uniform convergence results hold for the mollifier smoothing $F_{\varepsilon,U_\varepsilon}$.
In Section \ref{secao convergencia conexoes curvatura}, we prove that if $F$ is a Finsler structure, then the Chern connection, the Cartan connection, the Hashiguchi connection, the Berwald connection and the flag curvature of $F_\varepsilon$ converges uniformly on compact subsets to the correspondent objects of $F$ when $\varepsilon$ goes to zero.
In Section \ref{secao exemplos}, we apply the mollifier smoothing in order to study the curvature of a family of piecewise smooth Riemannian metrics.
In the two-dimensional case, we show that the total curvature can be concentrated on a subset of measure zero of $M$.
In Section \ref{secao obs finais}, we make final comments and we make suggestions for future works.

Most of this work was developed during the PhD of the second author under the supervision of the first author at State University of Maring\'a, Brazil.
The second author was supported by a CAPES PhD fellowship.
The authors would like to thank professors Bruno Mendon\c ca Rey dos Santos, Josiney Alves de Souza, Lino Anderson da Silva Grama and Patricia Hernandes Baptistelli for their valuable suggestions.

\section{Preliminaries}
\label{preliminares}

In this section, we present the preliminary subjects used for the development of this work. 
It is divided in three subsections: Convolution and smoothing, asymmetric norms and Finsler geometry.
As far as we know, Proposition \ref{funcao convergencia uniforme} and Theorem \ref{control minkowski boundary} are new (although they are intuitive and not difficult), and their proofs are placed in this section for the sake of convenience.  

\subsection{Convolution and smoothing}

In this subsection we present the theory of mollifier smoothings of continuous functions in Euclidean spaces. 
For the sake of simplicity, we restrict the presentation of the theory only for the continuous case, but the mollifier smoothing can be defined more generally on locally integrable functions. 
For more details of this topic and other topics presented in this subsection, see \cite{Evans}. 

\begin{defi}
\
\begin{enumerate}
\item A vector of the form $\alpha=(\alpha_1,\ldots,\alpha_n),$ where each component $\alpha_i$ is a  non-negative integer, is called a {\em multiindex} of order
\[
|\alpha|=\alpha_1+\cdots+\alpha_n;
\]
\item Given a  multiindex $\alpha,$ define
\[
D^\alpha f(x):= \frac{\partial^{|\alpha|}f(x)}{\partial (x^1)^{\alpha_1}\cdots\partial (x^n)^{\alpha_n}},
\]
where $f$ is a smooth real valued function defined in an open subset of ${\mathbb R^n}.$
\end{enumerate}
\end{defi}

Let $U$ be an open subset of ${\mathbb R}^n,$ $\varepsilon>0$ and $U_{\varepsilon} := \{x\in U \ | \ \dist(x,\partial U)>\varepsilon\}$, where $\partial U$ is the boundary of $U$.

\begin{defi} 
\label{define mollifier}

\

\begin{itemize}
\item[(i)] The {\em standard mollifier} $\eta \in \mathit{C}^{\infty}({\mathbb R}^n)$ is defined by
\begin{eqnarray*}
\eta(x) & :=& \begin{cases}

             C\exp\left(\frac{1}{ \|x\|^2-1}\right),  \quad & {\rm{if  \quad }}\|x\|<1, \\

             0, & {\rm{if}} \quad \|x\|\geq 1,

       \end{cases}
\end{eqnarray*}
where the constant $C> 0$ is chosen so that $\int_{{\mathbb R}^n}\eta \ dx=1;$

\item[(ii)]  For  each $\varepsilon>0,$ define
\begin{equation}
\label{epsilon-mollifier}
\eta_{\varepsilon}(x) := \frac{1}{\varepsilon^n}\eta\left(\frac{x}{\varepsilon}\right).
\end{equation}
\end{itemize}
\end{defi} 
Notice that
\[
\int_{{\mathbb R}^n}\eta_{\varepsilon} \ dx = \int_{{\mathbb R}^n}\eta \ dx=1  \quad {\rm and} \quad {\rm{supp}}(\eta_{\varepsilon})=B[0,\varepsilon],
\]
where $\supp (\eta_{\varepsilon})$ stands for the support of $\eta_{\varepsilon}$.

\begin{defi} If $f:U \rightarrow {\mathbb R}$ is a continuous function, its {\em mollifier smoothing} $ \eta_{\varepsilon}*f$ is the convolution of $f$ and $\eta_\varepsilon$ in $U_\varepsilon$, that is,
\[
(\eta_{\varepsilon}*f)(x)  := \int_U\eta_{\varepsilon}(x-y)f(y)dy = \int_{B[0,\varepsilon]}\eta_\varepsilon(y)f(x-y)dy,
\]
for every $x\in U_\varepsilon$.
\end{defi}

We are going towards Proposition \ref{funcao convergencia uniforme}, which states the uniform convergence of ``partial'' mollifier smoothings on compact subsets.
Convolutions (\ref{suaviza vertical}) and (\ref{suavizacao horizontal local}) are instances of this type of mollifier smoothing.

\begin{defi}
Let $k\in \{1, \ldots, n-1\}$ be an integer number and $\varepsilon >0$. 
Decompose $\mathbb R^n$ as $\mathbb R^k \times \mathbb R^{n-k}$ and denote its variables by $x=(x^1, \ldots, x^k)$ and $y = (x^{k+1}, \ldots, x^n)$. 
Let $U$ be an open subset of $\mathbb R^k$ and $V$ be an open subset of $\mathbb R^{n-k}$. 
Consider a continuous function $f: U \times V \rightarrow \mathbb R$ and let $\eta_\varepsilon : B(0,\varepsilon) \subset \mathbb R^k \rightarrow \mathbb R$ as in (\ref{epsilon-mollifier}). 
The mollifier smoothing $(\eta_\varepsilon \ast_1 f): U_\varepsilon \times V \rightarrow \mathbb R$ of $f$ with respect to the first $k$ variables is defined as
\[
(\eta_\varepsilon \ast_1 f) (x,y) = \int \eta_\varepsilon (z) f(x - z, y)dz.
\]
The mollifier smoothing $(\eta_\varepsilon \ast_2 f)$ with respect to the last $n-k$ variables is defined analogously.
\end{defi}

\begin{pro}
\label{funcao convergencia uniforme}
Let $U$ be an open subset of $\mathbb R^k$, $V$ be an open subset of $\mathbb R^{n-k}$ and $f: U \times V \rightarrow \mathbb R$ be a continuous function. Then $(\eta_\varepsilon \ast_1 f)$ converges uniformly to $f$ on compact subsets of $U \times V$. 
The same result holds for $(\eta_\varepsilon \ast_2 f)$.
\end{pro}

\begin{proof}
Let $K_{U\times V}$ be a compact subset of $U \times V$.
In what follows, we always suppose that $\varepsilon \in (0, \dist (K_{U\times V}, \partial (U \times V)))$.
Fix $(x_0, y_0) \in K_{U \times V}$ and $\delta > 0$.
Let $\tilde U_0 \subset U$ be a neighborhood of $x_0$ and $V_0 \subset V$ be a neighborhood of $y_0$ such that 
\begin{equation}
\label{define u0 til v0}
\vert f(x,y) - f(x_0,y_0)\vert < \delta / 2
\end{equation}
for every $(x,y) \in \tilde U_0 \times V_0$.
Let $U_0$ be a neighborhood of $x_0$ with compact closure such that $\overline U_0 \subset \tilde U_0$.
We claim that there exist $\varepsilon>0$ such that
\begin{equation}
\label{objetivo uniforme 2}
\vert (\eta_\varepsilon \ast_1 f) (x,y) - f(x,y)\vert < \delta
\end{equation}
for every $(x,y) \in U_0 \times V_0$.

Set $\varepsilon^\prime = \dist (\partial U_0, \partial \tilde U_0)>0$.
If $\varepsilon\in (0, \varepsilon^\prime)$, then
\begin{equation}
\label{diferencao f x0 y0 e f til}
(\eta_\varepsilon \ast_1 f) (x,y) 
= \int \eta_\varepsilon (z) f(x - z,y) dz \in \left( f(x_0, y_0) -\frac{\delta}{2}, f(x_0, y_0) + \frac{\delta}{2}\right)
\end{equation}
for every $(x,y) \in U_0 \times V_0$ due to the choice of $\varepsilon$ and (\ref{define u0 til v0}).
Therefore (\ref{objetivo uniforme 2}) follows for every $\varepsilon \in (0, \varepsilon^\prime)$ and $(x,y) \in U_0 \times V_0$ due to (\ref{define u0 til v0}) and (\ref{diferencao f x0 y0 e f til}).

In order to prove that $(\eta_\varepsilon \ast_1 f)$ converges uniformly to $f$ on $K_{U \times V}$, it is enough to cover $K_{U \times V}$ by a finite number of open subsets of type $U_0 \times V_0$ and choose the minimum of all $\varepsilon^\prime$.

The proof for $(\eta_\varepsilon \ast_2 f)$ is analogous.
\end{proof}

\begin{lem}
\label{derivada zeta vertical continua} 
Let $U$ be an open subset of $\mathbb R^k$, $V$ be an open subset of $\mathbb R^{n-k}$ and $f: U \times V \rightarrow \mathbb R$ be a continuous function.
Let $\varepsilon > 0$ and $\alpha$ be a multiindex with respect to the $k$ first variables of $U \times V$.
Let $(x_0,y_0) \in U \times V$ such that $\dist_{\mathbb R^n} ((x_0,y_0), \partial (U \times V)) > \varepsilon$.
Then $D^\alpha (\eta_\varepsilon \ast_1 f)$ is continuous at $(x_0,y_0)$.
The same result holds for $D^\alpha (\eta_\varepsilon \ast_2 f)$ if $\alpha$ is a multiindex with respect to the last $n-k$ variables.
\end{lem}

\begin{proof}
We will prove that for every $\delta>0$, there exist a $\mu >0$ such that if $\dist_{\mathbb R^n}((x,y),$ $(x_0,y_0)) < \mu$, then 
\[
\left\vert D^\alpha (\eta_\varepsilon \ast_1 f) (x,y) - D^\alpha (\eta_\varepsilon \ast_1 f) (x_0,y_0) \right\vert < \delta.
\]

Notice that $D^\alpha \left( \eta_\varepsilon \ast_1 f\right)$ is defined on a neighborhood $U_0 \times V_0$ of $(x_0,y_0)$ with compact closure because $\dist_{\mathbb R^n} ((x_0,y_0), \partial (U \times V)) > \varepsilon$.
Moreover we can suppose that $\dist_{\mathbb R^n} (\bar U_0 \times \bar V_0,\partial (U \times V)) > \varepsilon$.
Denote
\[
B \left[ \bar U_0 \times \bar V_0 , \varepsilon \right] = \left\{(x,y) \in \mathbb R^k \times \mathbb R^{n-k}; \dist_{\mathbb R^n}((x,y),\bar U_0 \times \bar V_0) \leq \varepsilon \right\}  \subset U \times V.
\]

For every $\delta > 0$, there exist a $\mu >0$ such that 
\[
\vert f(x_1,y_1) - f(x_2,y_2)\vert < \frac{\delta}{\int \left\vert \left( D^\alpha \eta_\varepsilon \right) (z) \right\vert dz}
\] 
whenever $(x_1,y_1)$, $(x_2,y_2) \in B[\bar U_0 \times \bar V_0, \varepsilon]$ and $\dist_{\mathbb R^n}((x_1, y_1),(x_2,y_2))< \mu$.

Suppose that $(x,y) \in U_0 \times V_0$ is such that $\dist_{\mathbb R^n}((x,y),(x_0,y_0)) < \mu$.
Then
\begin{eqnarray}
& & D^\alpha (\eta_\varepsilon \ast_1 f)(x,y) = D^\alpha \left( \int \eta_\varepsilon (x-z) f(z,y) dz \right)  \nonumber \\ 
& = &  \int \left( D^\alpha \eta_\varepsilon \right) (x-z) f(z,y) dz
= \int \left( D^\alpha \eta_\varepsilon\right) (z) f(x-z,y) dz \nonumber
\end{eqnarray}
and
\begin{eqnarray}
& & \left\vert D^\alpha (\eta_\varepsilon \ast_1 f) (x,y) - D^\alpha (\eta_\varepsilon \ast_1 f) (x_0,y_0) \right\vert \nonumber \\ 
& & \leq \int \left\vert \left( D^\alpha \eta_\varepsilon \right) (z) \right\vert \left\vert f(x-z,y) - f(x_0 - z,y_0) \right\vert dz < \delta. \nonumber 
\end{eqnarray}
Therefore $D^\alpha (\eta_\varepsilon \ast_1 f)$ is continuous at $(x_0,y_0)$.

The proof when $\alpha$ is a multiindex with respect to the last $n-k$ variables is analogous.
\end{proof}

\subsection{Asymmetric norms}

In this subsection we present the theory of asymmetric norms and Minkowski norms which are used in this work.
The references for this subsection are \cite{BaoChernShen} and \cite{Cobzas}. 
Issues related to convex analysis can be found in \cite{RockafellarTyrrell} and \cite{SunYuan}.

\begin{defi}
\label{def asymmetric norm}
An {\em asymmetric norm} on a real vector space $\mathbb V$ is a non-negative function $F:\mathbb V \rightarrow \mathbb R$ that satisfies the following conditions:
\begin{enumerate}
\item $F(y)=0$ only if $y=0$;
\item $F(\mu y)=\mu F(y)$ for every $\mu \in [0,\infty)$ and $y\in \mathbb V$;
\item $F(y+z) \leq F(y) + F(z)$ for every $y,z \in \mathbb V$.
\end{enumerate}
\end{defi}

\begin{obs}
If an asymmetric norm $F:\mathbb V \rightarrow \mathbb R$ satisfies $F(y)=F(-y)$ for every $y \in \mathbb V$, then $F$ is a norm.
\end{obs}

\begin{defi}
Let $\mathbb V$ be a real vector space endowed with an asymmetric norm $F$.
The {\em open ball} centered at $y$ and radius $r$ is the subset
\[
B_F(y,r) = \{z \in \mathbb V;F(z-y) < r\}.
\]
The {\em closed ball} centered at $y$ and radius $r$ is the subset
\[
B_F[y,r] = \{z \in \mathbb V;F(z-y) \leq r\}.
\]
The {\em sphere} centered at $y$ and radius $r$ is the subset
\[
S_F[y,r] = \{z \in \mathbb V;F(z-y) = r\}.
\]
\end{defi}

\begin{obs}
It is straightforward that every real valued function $F: \mathbb V \rightarrow \mathbb R$ that satisfies Items (2) and (3) of Definition \ref{def asymmetric norm} is convex.
In particular, if $F$ is an asymmetric norm, then $B_F(y,r)$ and $B_F[y,r]$ are convex.
\end{obs}

A special case of asymmetric norm are Minkowski norms, which play a fundamental role in Finsler geometry.

\begin{defi}
\label{norma de Minkowski} 
A function $F:\mathbb V \rightarrow \mathbb R$ is a Minkowski norm if 
\begin{itemize}
\item[(i)] $F$ is smooth in $\mathbb V\backslash\{0\}$;
\item[(ii)] $F(\mu y) = \mu F(y)$ for every $\mu > 0$ and $y\in \mathbb V$;
\item[(iii)] If $(y^1, \ldots, y^n)$ is a coordinate system of $\mathbb V$ with respect to a basis of $\mathbb V$, then the  $n\times n$ hessian matrix
\[
\left( g_{ij(y)} \right) := \left( \left[\frac{1}{2} F^2 (y) \right]_{y^iy^j} \right), \hspace{5mm} i,j=1, \ldots, n,
\] 
is positive definite for every $y \in \mathbb V\backslash \{0\}$, where the subscript $y^iy^j$ stands for the partial derivatives with respect to $y^i$ and $y^j$.
\end{itemize}
\end{defi}

\begin{teo}
Let $F$ be a Minkowski norm on $\mathbb V$.
Then $F(y) > 0$ if $y\neq 0$ and $F(y+z)\leq F(y) + F(z)$ for every $y,z \in \mathbb V$.
In particular, $F$ is an asymmetric norm.
\end{teo}

The following theorem is very intuitive and characterizes Minkowski norms. 
We give its proof for the sake of completeness.

\begin{teo}
\label{control minkowski boundary}
Consider $\mathbb R^n$ endowed with the Euclidean  metric.
Let $F:\mathbb R^n \rightarrow \mathbb R$ be an asymmetric norm such that its restriction to $\mathbb R^n\backslash \{0\}$ is smooth.
For $r>0$, consider $S_F[0,r]$ endowed with the Riemannian metric induced by its embedding in $\mathbb R^n$ and oriented by its inward normal vector field. 
Then the following statements are equivalent:
\begin{enumerate}
\item $F$ is a Minkowski norm;
\item $\Hess F$ has rank $n-1$ on $\mathbb R^n\backslash \{0\}$;
\item For every $r>0$, $S_F[0,r]$ is locally a graph of a function $\psi:U \subset \mathbb R^{n-1} \rightarrow \mathbb R$, where $\mathbb R^{n-1}$ is a totally geodesic submanifold of $\mathbb R^n$, $\psi(0)=0$, $d\psi_0 \equiv 0$ and $\Hess \psi$ is positive definite;
\item $S_F[0,1] \hookrightarrow \mathbb R^n$ has positive sectional curvature;
\item All the principal curvatures $\kappa_1, \ldots, \kappa_{n-1}$ of $S_F[0,1] \hookrightarrow \mathbb R^n$ are strictly positive.
\end{enumerate}
\end{teo}

\begin{proof}
Let $y \in \mathbb R^n \backslash \{0\}$ and denote $r = F(y)$.
Consider an Euclidean coordinate system $(z^1, \ldots, z^n)$ of $\mathbb R^n$ centered at $y$ such that $S_F[0,r]$ is represented as a graph of a function $\psi:U \subset \mathbb R^{n-1} \rightarrow \mathbb R$ given by
\[
z^n = \psi (z^1, \ldots, z^{n-1} ) = \sum_{i=1}^{n-1} \frac{1}{2}\kappa_i (z^i)^2 + O(\Vert z \Vert^3)
\]
in a neighborhood of $y$, where $\kappa_1, \ldots, \kappa_{n-1}$ are the principal curvatures of $S_F[0,r]$ at $y$ (see \cite{LeeCurvature}).
Notice that $\kappa_i \geq 0$ for every $i=1, \ldots, n-1$ due to the orientation of $S_F[0,r]$.
The matrix of $\Hess \psi(0)$ with respect to the coordinates $(z^1, \ldots, z^{n-1})$ is given by
\[
\left[ \Hess \psi (0)\right]_{(z^1, \ldots, z^{n-1})} = 
\left[ 
\begin{array}{cccc}
\kappa_1 & 0 & \ldots & 0 \\
0 & \kappa_2 & & 0 \\
\vdots & & \ddots & \vdots \\
0 & 0 & \ldots & \kappa_{n-1}
\end{array}
\right].
\]
This settles Item (3) $\Leftrightarrow$ Item (5).

Let $\nu$ be the unit normal vector field of $S_F[0,1]$ pointed towards $B_F[0,1]$. 
The shape operator $A_\nu(y): T_yS_F[0,r] \rightarrow T_yS_F[0,r]$ of the isometric embedding $S_F[0,1]$ $\hookrightarrow \mathbb R^n$ with respect to $\nu$ is also given by 
\[
[A_\nu(y)]_{(z^1, \ldots, z^n)} =  
\left[ 
\begin{array}{cccc}
\kappa_1 & 0 & \ldots & 0 \\
0 & \kappa_2 & & 0 \\
\vdots & & \ddots & \vdots \\
0 & 0 & \ldots & \kappa_{n-1}
\end{array}
\right] 
\]
(see \cite{doCarmo2}).
Gauss equation states that the sectional curvature of a two-dimensional subspace $\chi = \spann \{v,w\}$ of $T_yS_F[0,1]$ is given by
\[
K(\chi) = \frac{\left< A_\nu(y) v,v \right>\left< A_\nu(y) w,w \right> - \left< A_\nu(y) v,w \right>^2}{\left< v,v \right>\left< w, w \right> - \left< v, w \right>^2}.
\]
Observe that $\left< \cdot, \cdot \right>^\prime = \left< A_\nu (y) \cdot, \cdot \right>$ is an inner product iff all the principal curvatures are strictly positive.
In addition it is not difficult to see that $\left< v, v\right>^\prime \left< w, w\right>^\prime - \left< v, w \right>^\prime > 0$ for every $\chi$ iff $\left<  \cdot, \cdot\right>^\prime$ is an inner product.
This settles Item (4) $\Leftrightarrow$ Item (5). 

Now we prove the equivalence among Items (1), (2) and (3).
Consider the coordinate system $(z^1, \ldots, z^n)$ as in the proof of Item (3) $\Leftrightarrow$ Item (5). 
Then $\partial F^2 / \partial z^n(0)$ and $\partial F/ \partial z^n(0)$ are strictly negative.
We have that 
\begin{equation}
\label{esfera local}
F^2(z^1, \ldots, z^{n-1}, \psi(z^1, \ldots, z^{n-1}))=r^2.
\end{equation}
Let $i,j\in \{1,\ldots n-1\}$.
Calculating the derivative of (\ref{esfera local}) with respect to $z^i$ and $z^j$ and evaluating at the origin, we have that

\begin{equation}
\label{F e phi no kernel}
\frac{\partial^2 F^2}{\partial z^j \partial z^i}(0) 
=  - \frac{\partial F^2}{\partial z^n}(0)
\frac{\partial^2 \psi}{\partial z^j\partial z^i}(0).
\end{equation}
We also have that
\begin{equation}
\label{deriva F2 em termos de F}
\frac{\partial^2 F^2}{\partial z^i \partial z^j}(0) 
= 2 F(0)\frac{\partial^2 F}{\partial z^i \partial z^j}(0) \text{ if }(i,j) \neq (n,n)
\end{equation}
and
\begin{equation}
\label{derivada F2 em zn}
\frac{\partial^2 F^2}{\partial (z^n)^2}(0) 
= 2 F(0) \frac{\partial^2 F}{\partial (z^n)^2}(0) + 2 \left( \frac{\partial F}{\partial z^n}\right)^2(0).
\end{equation}
In order to see that Item (2) is equivalent to Item (3), observe that $\Hess \psi (0)$ is positive definite iff $\Hess F^2(0)$ is positive definite when restricted to the tangent space of $S_F[0,r]$ due to (\ref{F e phi no kernel}).
But this last statement holds iff $\Hess F(0)$ has rank $n-1$ due to (\ref{deriva F2 em termos de F}) (Observe that the rank of $\Hess F(0)$ is less than or equal to $n-1$ because $F$ is a norm).
This settles the equivalence between Item (2) and Item (3).

Now we prove the equivalence between Item (1) and Item (3). 
If $F$ is a Minkowski norm, then (\ref{F e phi no kernel}) implies that $\Hess \psi(0)$ is also positive definite, what settles Item (1) $\Rightarrow$ Item (3).
For the inverse implication suppose that $\Hess \psi (0)$ is positive definite.
Let us prove that $F$ is a Minkowski norm.
Equations (\ref{deriva F2 em termos de F}) and (\ref{derivada F2 em zn}) implies that
\[
\text{Hess}(F^2)(0) = 2F(0)\text{Hess} F(0) + \rho,
\]
where the matrix of $\rho$ with respect to $(\partial / \partial z^1, \ldots, \partial / \partial z^n)$ is given by
\[
\rho = 
\left[
\begin{array}{cccc}
0 & \ldots & 0 & 0 \\
\vdots & \ddots &  & \vdots \\
0 &  & 0 & 0 \\
0 & \ldots & 0 & 2 \left(\frac{\partial F}{\partial z^n} (0) \right)^2 \\
\end{array}
\right].
\]
Observe that $\Hess F(0)$ and $\rho$ are positive semidefinite.

If $\xi \not\in \text{span}\{\partial / \partial z^1, \ldots, \partial / \partial z^{n-1}\}$, then it has non-zero component in $\partial / \partial z^n$ and 
\[
\text{Hess}(F^2)(0)(\xi,\xi) > 0
\]
due to $\rho$.

If $\xi \in \text{span}\{\partial / \partial z^1, \ldots, \partial / \partial z^{n-1}\}\backslash 0$, then Equation (\ref{F e phi no kernel}) are in place with
\[
\frac{\partial F^2}{\partial z^n}(0) < 0.
\]
Therefore $\text{Hess}(F^2)(0)(\xi,\xi)>0$, what proves that $F$ is a Minkowski norm.
\end{proof}

\subsection{Finsler geometry}

In this subsection we introduce elements of Finsler geometry.
We present the fundamental tensor, the Cartan tensor, several linear connections and the flag curvature. 
A reference for Finsler geometry is \cite{BaoChernShen}.
For tensor algebra formalism, see \cite{DubrovinFomenkoNovikov1}.
A reference for linear connections on vector bundles is \cite{Darling}.

Let $\mathbb V$ be a real vector space and $\mathcal B=\{e_1, \ldots,$ $e_n\}$ is a basis of $\mathbb V$. 
Let $\mathbb V^\ast$ be the dual vector space of $\mathbb V$ and denote by $\mathcal B^\ast = \{e^1, \ldots, e^n\}$ the dual basis of $\mathcal B$.
A $p$ times contravariant and $q$ times covariant tensor $T$ on $\mathbb V$ is represented by
\[
T^{i_1 \ldots i_p}_{j_1 \ldots j_q}e_{i_1}\otimes \ldots \otimes e_{i_p} \otimes e^{j_1} \otimes \ldots \otimes e^{j_q}
\]
or simply by $T^{i_1 \ldots i_p}_{j_1 \ldots j_q}$.
When the $\mathbb V$ is endowed with a inner product $g=(g_{ij})$, then $(g^{ij})$ represent the inverse tensor of $g$ and the operation of raising and lowering indices are denoted by $\rho^j=g^{ji}\rho_i$ and $\xi_j = g_{ij}\xi^i$ respectively.


Let $(M,F)$ be a Finsler manifold and consider the natural projection $\pi: TM\setminus 0 \rightarrow M$. The pulled-back tangent bundle  
\[
\pi^* TM = \bigcup_{x\in M} \left( \{x\} \times (T_xM\setminus\{0\})\times T_xM  \right)
\]
and the pulled-back cotangent bundle  
\[
\pi^* T^*M = \bigcup_{x\in M} \left( \{x\} \times (T_xM\setminus\{0\})\times T_x^*M  \right)
\]
are vector bundles over the slit tangent bundle $TM\setminus 0$.

Let  $(x^i)$ be a coordinate system of an open subset $U \subset M$ and $((x^i),(y^i))$ be the natural coordinate system on $TU$ induced by $(x^i)$. 
The sections $\left\{ \frac{\partial}{\partial x^i} \right\}$ and $ \{ dx^i\}$ on $\pi^*TU$ and $\pi^*T^*U$ are respectively given by
\begin{eqnarray*}
\frac{\partial}{\partial x^i} : TU\setminus 0 \rightarrow \pi^*TU, & \ & \frac{\partial}{\partial x^i}(x,y) = \left( x, y, \frac{\partial}{\partial x^i}\Big\vert_x \right),\\
d x^i : TU\setminus 0 \rightarrow \pi^*T^*U, & \ & d x^i(x,y) = \left( x, y, d(x^i)_x \right).
\end{eqnarray*}
These sections are defined locally in $x$ and globally in $y$.

\begin{defi} 
\label{define secao distinguida} 
The {\em distinguished section} $\ell$ is a section of $\pi^*TM$ defined by
\begin{equation}
\ell \ = \ \ell_{(x,y)} \ := \ \frac{y^i}{F(x,y)}\frac{\partial}{\partial x^i}\Big\vert_x \ = \ \frac{y^i}{F}\frac{\partial}{\partial x^i} \ =: \ \ell^i \frac{\partial}{\partial x^i}.
\label{define distinguished} 
\end{equation} 
\end{defi}

\begin{defi} 
\label{define tensor fundamental} 
The {\em fundamental tensor} is a section of inner products on $\otimes^2 \pi^\ast T^\ast  M$ given locally by
\[
g_{ij(x,y)} = \frac{1}{2}\left( \frac{\partial^2 F^2}{\partial y^i \partial y^j}\right)(x,y).
\]
\end{defi}

\begin{pro} \label{g ij redimensionamento} Let $F:TM \rightarrow [0,\infty)$ be a Finsler structure. Given $(x,y) \in TM\setminus 0$, we have that
$g_{ij(x,y)}=g_{ij(x,\mu y)}$ for every $\mu > 0$. 
\end{pro}

\begin{proof}
See \cite{BaoChernShen}.
\end{proof}

\begin{defi} 
\label{define tensor de cartan} 
Let $(M,F)$ be a Finsler manifold. 
The {\em Cartan tensor} is symmetric section $A$ on $\otimes^3 \pi^\ast T^\ast M$ defined locally by 
\begin{equation}
\label{define componente tensor de cartan}
A_{(x,y)} = A_{ijk(x,y)} dx^i \otimes dx^j \otimes dx^k
:= \frac{1}{2} .F. \frac{\partial g_{ij}}{\partial y^k} dx^i \otimes dx^j \otimes dx^k.
\end{equation}
\end{defi}
In the literature, the tensor field defined locally as
\begin{equation}
\label{unormalized Cartan tensor}
C_{ijk}=\frac{1}{2}\frac{\partial g_{ij}}{\partial y^k}
\end{equation}
is also called the {\em Cartan tensor}.

The {\em formal Christoffel symbols} of the second kind are defined as 
\begin{equation} 
\gamma^i{}_{jk} := g^{is}\frac{1}{2}\left(\frac{\partial g_{sj}}{\partial x^k}-\frac{\partial g_{jk}}{\partial x^s}+\frac{\partial g_{ks}}{\partial x^j}  \right).
\label{definicao formal Christoffel simbols}
\end{equation}

The {\em nonlinear connection} is given by
\begin{equation}
N^i{}_j := \gamma^i{}_{jk} y^k- C^i{}_{jk}\gamma^k{}_{rs}y^ry^s,
\label{definicao non linear connection}
\end{equation}
where $C^i{}_{jk} = g^{is}C_{sjk}$. 

It is usual to consider the basis $\left\{ \frac{\delta}{\delta x^i}, F\frac{\partial}{\partial y^i}\right\}$ and $\left\{ dx^i, \frac{\delta y^i}{F}\right\}$ on $TU\setminus 0$ and $T^\ast U$ respectively, where
\begin{equation}
\label{define delta delta}
\frac{\delta}{\delta x^j} := \frac{\partial}{\partial x^j}-N^i{}_j \frac{\partial}{\partial y^i} 
\end{equation}
and 
\[
\delta y^i:= dy^i+N^i{}_j dx^j.
\]

The Chern connection is characterized by the following theorem.

\begin{teo} 
\label{teorema conexao de Chern}
Let $(M,F)$ be a Finsler manifold. The pulled-back bundle  $\pi^*TM$ admits a unique linear connection $\nabla$, called the {\em  Chern connection}, such that its connection forms $\omega_j^{\ i}$, defined by
\[
\nabla_X \frac{\partial}{\partial x^j} = \omega_j{}^i(X) \frac{\partial}{\partial x^i},
\]
are locally characterized by the following structural equations:
\begin{itemize}
\item {\bf Torsion freeness:}
\[
d(dx^i)-dx^j\wedge \omega_j{}^i = - dx^j \wedge \omega_j{}^i=0.
\]
\item {\bf Almost $g$-compatibility:} 
\[
dg_{ij}- g_{kj} \omega_i{}^k-g_{ik}\omega_j{}^k = 2C_{ijs}\delta y^s. 
\]
\end{itemize}
\end{teo}

\begin{proof}
See \cite{BaoChernShen}.
\end{proof}

The torsion freeness is equivalent to the absence of $dy^k$ in  $\omega_j^{\ i},$ that is,
\[
\omega_j{}^i = \Gamma^i{}_{jk}dx^k,
\]
together with the symmetry
\[
\Gamma^i{}_{jk}=\Gamma^i{}_{kj}.
\]
The almost metric-compatibility  implies that
\begin{equation}
\Gamma^l{}_{jk}= \gamma^l{}_{jk}-g^{li}\left( C_{ijs}N^s{}_k- C_{jks}N^s{}_i+ C_{kis}N^s{}_j \right),
\label{Gamma almost metric compatibility}
\end{equation}
or equivalently
\[
\Gamma^i{}_{jk}= \frac{g^{is}}{2}\left(\frac{\delta g_{sj}}{\delta x^k}-\frac{\delta g_{jk}}{\delta x^s}+\frac{\delta g_{ks}}{\delta x^j} \right).
\]

Denote 
\[
\dot A = \nabla_{\ell^i \frac{\delta}{\delta x^i}}A,
\] 
where $A$ is the Cartan tensor. Then
\begin{equation}
\label{define a dot}
\dot A^i{}_{jk} = \frac{1}{2}(\sigma^i)_{y^jy^k} - \Gamma^i{}_{jk}
\end{equation}
where
\begin{equation}
\label{definicao Gi}
\sigma^i = \gamma^i{}_{jk}y^j y^k
\end{equation}
(see \cite{BaoChernShen}).

The curvature 2-forms of the Chern connection are defined by
\[
\Omega_j{}^i:=  d\omega_j{}^i - \omega_j{}^k\wedge \omega_k{}^i. 
\]
Writing the $2$-forms $\Omega_j^{\ i}$ in terms of the basis  $\left\{dx^k,\frac{\delta y^k}{F}\right\},$ we have that
\[
\Omega_j{}^i = \frac{1}{2}R_j{}^i{}_{kl} dx^k\wedge dx^l+P_j{}^i{}_{kl} dx^k \wedge \frac{\delta y^l}{F}+\frac{1}{2}Q_j{}^i{}_{kl} \frac{\delta y^k}{F} \wedge \frac{\delta y^l}{F}.
\]
The objects $R,P$ and $Q$ are respectively the $hh$-, $hv$-, $vv$-curvature tensors of the Chern connection and they can be chosen as
\begin{eqnarray}
R_j{}^i{}_{kl} &=& \frac{\delta \Gamma^i{}_{jl}}{\delta x^k}-\frac{\delta \Gamma^i{}_{jk}}{\delta x^l}+ \Gamma^i{}_{hk}\Gamma^h{}_{jl}-\Gamma^i{}_{hl}\Gamma^h{}_{jk}, \label{Rjikl hh}\\
P_j{}^i{}_{kl} &=& -  F \frac{\partial \Gamma^i{}_{jk}}{\partial y^l}, \nonumber \\
Q_j{}^i{}_{lk} & = & 0. \nonumber
\end{eqnarray}
Notice that
\[
R_j{}^i{}_{kl} = - R_j{}^i{}_{lk}.
\]

Now we define the flag curvature. 
A {\em flag} at $x \in M$ consists of a nonzero vector $y \in T_xM,$ which is called the {\em flagpole}, and a vector $z:=z^i\frac{\partial}{\partial x^i} \in T_xM$ which is transversal to $y$.  
The flag curvature is the correspondence which associates to each $x \in M$ and each flag $(y,z) $ in $T_xM$  the number
\[
K(y,z):=\frac{z^i(y^jR_{jikl}y^l)z^k}{g(y,y)g(z,z)-[g(y,z)]^2},
\]
where 
\begin{equation}
R_{jikl} = g_{im}R_j{}^m{}_{kl}.
\label{Rjikl abaixa}
\end{equation}
$K(y,z)$ depends only on $y$ and $\spann\{y,z\}$.

We end this section defining other connections:

\begin{defi}
\label{Cartan Hasiguchi Berwarld} The connection forms of the following connections on $\pi^\ast TM$ are defined by:
\begin{enumerate}
\item Cartan connection: $\omega_j{}^i + C^i{}_{jk}.\delta y^k$;
\item Hashiguchi connection: $\omega_j{}^i + C^i{}_{jk}.\delta y^k + \dot A^i{}_{jk}.dx^k$;
\item Berwald connection: $\omega_j{}^i + \dot A^i{}_{jk}.dx^k$.
\end{enumerate}
\end{defi}

\begin{obs}
The flag curvature for the Chern connection, Cartan connection, Hashiguchi connection and Berwald connection are the same.
\end{obs}


\section{Mollifier smoothing of convex functions}
\label{secao mollifier smoothing}

In this section we study the mollifier smoothing of convex functions $f:\mathbb R^n \rightarrow \mathbb R$.

\begin{lem} 
\label{suavizacao mollifier e convexa}  
Let $f:{\mathbb R}^n\rightarrow {\mathbb R}$ be a convex function. Then the mollifier smoothing
\[
(\eta_{\varepsilon}*f)(x) = \int_{{\mathbb R}^n} \eta_\varepsilon(x-y)f(y)dy = \int_{{\mathbb R}^n} \eta_\varepsilon(y)f(x-y)dy
\]
is also  a convex function.
\end{lem}

\begin{proof}
It is straightforward from the convexity of $f$.
\end{proof}

\begin{obs}
\label{formato eta}
The bell shape of the graph of $\eta: \mathbb R \rightarrow \mathbb R$ is described below:
\begin{itemize}
\item $\eta$ is identically zero in $\mathbb R \backslash (-1,1)$ and it is strictly positive in $(-1,1)$;
\item Its derivative is strictly positive in $(-1,0)$ and it is strictly negative in $(0,1)$;
\item Its second derivative is strictly positive in $\left(-1,-\sqrt[4]{\frac{1}{3}}\right) \cup\left(\sqrt[4]{\frac{1}{3}},1\right)$ and strictly negative in $\left(-\sqrt[4]{\frac{1}{3}},\sqrt[4]{\frac{1}{3}}\right)$.
\end{itemize}

It is straightforward that $\eta_\varepsilon$ has the same qualitative behavior as $\eta$.
\end{obs}

The next lemma is important in order to guarantee that the Hessian $\left( \frac{\partial^2 F^2_\varepsilon}{\partial y^i \partial y^j} \right)$ is positive definite.


\begin{lem} 
\label{condicao para curvatura} 
If $f:{\mathbb R} \rightarrow {\mathbb R}$ be a convex function, then $(\eta_{\varepsilon}*f){''}(x)\geq 0$ for every $x \in \mathbb R$. Moreover $(\eta_{\varepsilon}*f)''(x)=0$ if and only if $f$ is affine in $(x-\varepsilon,x+\varepsilon).$ 
\end{lem}

\begin{proof}
Lemma \ref{suavizacao mollifier e convexa}  assures that $(\eta_{\varepsilon}*f)$ is convex and $(\eta_{\varepsilon}*f){''}(x)\geq 0.$ 
It remains to prove that $(\eta_{\varepsilon}*f)''(x)=0$ if and only if $f$ is affine in $(x-\varepsilon,x+\varepsilon)$.
 
Let $y\mapsto (y,ay+b)$ be a parameterized tangent  line of the graph of $f$ at $x$ (which may not be unique). 
Define $h(y) = f(y) - ay - b$. Then $h$ is a non-negative convex function that attains an absolute minimum at $h(x) = 0$.
The restriction $h|_{(-\infty,x)}$ is non increasing and $h|_{(x,\infty)}$ is non decreasing.
Moreover, if $h(z)>0$ and $z>x,$ then $h|_{(z,\infty)}$ is strictly increasing. Similarly $h$ is strictly decreasing in $(-\infty,z)$ if $h(z)>0$ and $z<x$.

Notice that 
\[
(\eta_{\varepsilon}*f)'' (x) = \int_{{\mathbb R}}\eta''_{\varepsilon}(x-y)[f(y)-ay-b]dy \nonumber
\]
because  $\int_{{\mathbb R}}\eta''_{\varepsilon}(x-y)[-ay-b]dy=0$. In particular if $f$ is affine in $(x-\varepsilon,x+\varepsilon)$, then $f(y)=ay+b$ for every $y\in (x-\varepsilon,x+\varepsilon)$ and $(\eta_{\varepsilon}*f)'' (x)=0.$

Now suppose that $f$ is not affine in $(x-\varepsilon,x+\varepsilon)$. 
Let $-\iota$ and $ \iota$ be the inflection points of $\eta_\varepsilon$.
We write $(\eta_{\varepsilon}*f)'' (x)$ as the sum of
\begin{eqnarray*}
C_1 & = & \int_{x-\varepsilon}^{x-\iota}\eta''_\varepsilon(x-y)h(y)dy, \\
C_2 & = & \int_{x-\iota}^{x}\eta''_\varepsilon(x-y)h(y)dy, \\
C_3 & = & \int_{x}^{x+\iota}\eta''_\varepsilon(x-y)h(y)dy
\end{eqnarray*}
and
\begin{eqnarray*}
C_4 & = & \int_{x+\iota}^{x+\varepsilon}\eta''_\varepsilon(x-y)h(y)dy.
\end{eqnarray*}
As $h(x-\iota)\leq h(y)$ and  $\eta_\varepsilon''(x-y)\geq 0$ for every $ y\in[x-\varepsilon,x-\iota]$, then
\begin{equation} 
\label{C_1}
C_1 \geq \int_{x-\varepsilon}^{x-\iota}\eta''_\varepsilon(x-y)h(x-\iota)dy 
= h(x-\iota) \int_{x-\varepsilon}^{x-\iota}\eta''_\varepsilon(x-y)dy.
\end{equation}
Following the same reasoning, we have that $h(x-\iota)\geq h(y)$ and $\eta_\varepsilon''(x-y)\leq 0$ for every $ y\in[x-\iota,x]$ and
\[
C_2 \geq \int_{x-\iota}^{x}\eta''_\varepsilon(x-y)h(x-\iota)dy = h(x-\iota)\int_{x-\iota}^{x}\eta''_\varepsilon(x-y)dy.
\]
Analogously we obtain 
\[
C_3 \geq \int_{x}^{x+\iota}\eta''_\varepsilon(x-y) h(x+\iota)dy = h(x+\iota) \int_{x}^{x+\iota}\eta''_\varepsilon(x-y)dy
\]
and
\begin{equation} \label{C_4}
C_4 \geq \int_{x+\iota}^{x+\varepsilon}\eta''_\varepsilon(x-y)h(x+\iota)dy = h(x+\iota) \int_{x+\iota}^{x+\varepsilon}\eta''_\varepsilon(x-y)dy.
\end{equation}
As $f$ is not affine in $(x-\varepsilon,x+\varepsilon)$, then the inequality (\ref{C_1}) or (\ref{C_4}) is strict. Therefore 
\begin{eqnarray*}
 (\eta_{\varepsilon}*f)''(x)  & =& C_1+C_2+C_3+C_4  \\
 & >&  h(x-\iota) \int_{x-\varepsilon}^{x}\eta''_\varepsilon(x-y)dy + h(x+\iota) \int_{x}^{x+\varepsilon}\eta''_\varepsilon(x-y)dy \\
   & =& [h(x-\iota)+ h(x+\iota)]\int_{x-\varepsilon}^{x}\eta''_\varepsilon(x-y)dy \\
  & = & [h(x-\iota)+ h(x+\iota)] [-\eta'_\varepsilon(x-y)]\vert_{y=x-\varepsilon}^{y=x} \\
   & = & 0,
\end{eqnarray*}
what settles the lemma.
\end{proof}


\section{Growth rates of asymmetric norms on ${\mathbb R^n}$}

\label{section growth rates}

Let $F$ be an arbitrary asymmetric norm on ${\mathbb R^n}$ and $\|\cdot\|$ be its canonical norm. 
We use notations like $B_{\|\cdot\|}(y, r)$ and $B_F(y, r)$ in order to denote the open ball with center $y$ and radius $r$ with respect to $\|\cdot\| $ and $F$ respectively. 
The sphere and the closed ball with respect to $F$ will be denoted by $S_F[y, r]$ and $B_F[y, r]$ respectively and so on. 

An asymmetric norm $F$ in ${\mathbb R^n}$ is completely characterized by $S_F[0,1],$ which is the boundary of the convex subset $B_F[0,1].$ 
For every  $y \in S_F[0,1]$ there exist a supporting hyperplane $H_y$ of $B_F[0,1]$ at $y.$ 
In general $H_y$ is not unique. For a fixed $H_y,$ there exist a unique linear function  $L_y$ such that $L_y=F(y)=1$ on $H_y.$  
We are interested to find a positive lower bound for the growth rate of $F$ in almost radial directions. 
In order to do that, we use the relationship $F \geq L_y.$

Define
\begin{equation} 
\label{constante r}
r_M  =  \max_{y \in S_{\|\cdot\|}[0,1]} F(y) = \max_{y \in B_{\|\cdot\|}[0,1]} F(y)
\end{equation}
and
\begin{equation} 
\label{constante rho}
r_m  =  \min_{y \in S_{\|\cdot\|}[0,1]} F(y).
\end{equation}
Then
\begin{equation} 
\label{identidade entre max e min}
\max_{y \in S_{\|\cdot\|}[0,1]} F(y) = \min_{y \in S_{\|\cdot\|}\left[ 0,\frac{r_M}{r_m} \right]} F(y).
\end{equation}

We know that the directional derivative of $F$ at $v$ in the radial direction $\frac{v}{\|v\|}$ is $\frac{F(v)}{\|v\|},$ which can be bounded below by $r_m$ due to (\ref{constante rho}). The next lemma is a version of this result for the  ``derivative'' of $F$ with respect to $\frac{v}{\|v\|}$  in a neighborhood of $v.$ Of course $F$ is not necessarily differentiable, and we overcome this shortcoming using differences
instead of derivatives.

\begin{lem} \label{proposicao F em Rn}
Let $F$ be an asymmetric norm on ${\mathbb R^n}$ and $v \in {\mathbb R^n}$ be a nonzero vector. Set
\[
\varepsilon \in \left(0,\frac{r_m \Vert v\Vert}{4nr_M}\right],
\]
where $r_M$ and $r_m$ are defined by (\ref{constante r}) and (\ref{constante rho}) respectively. Then
\[
F\left(w+t\frac{v}{\|v\|}\right)-F(w) > \frac{r_m}{8}t
\]
for every $w\in B_{\|\cdot\|}[v,\varepsilon]$ and $t>0$.
\end{lem}

{\it Proof. }
Without loss of generality, we can suppose that $v=(0,\ldots,0, \Vert v\Vert).$ The general case can be reduced to this case after an orthogonal change of coordinates.

Denote $\tilde{\mathbb R}^{n-1} = \{(y^1, \ldots, y^n) \in \mathbb R^n;y^n=0\}$ and let $\tilde \pi: \mathbb R^n \rightarrow \tilde{\mathbb R}^{n-1}$ be the orthogonal projection.
Define
\[ 
D=\tilde{\mathbb R}^{n-1}\cap B_{\|\cdot\|}\left[0,\frac{r_m \Vert v\Vert}{2r_M}\right]
\]
and $E=B_{\|\cdot\|}[v,\varepsilon]$. 
Observe that $\tilde \pi(E) \subset D$.

Note that
\begin{equation} 
\max\limits_{y\in D} F(y) 
\leq \frac{r_m \Vert v\Vert}{2} 
 = \min\limits_{y \in S_{\|\cdot\|} \left[ 0, \frac{\Vert v\Vert}{2} \right] }F(y)  <  \min\limits_{y\in E}F(y), \label{desigualdades}
\end{equation}
where the equalities are due to (\ref{identidade entre max e min}) and the last inequality holds because $E\cap B_{\|\cdot\|} \left[ 0,\frac{\Vert v\Vert }{2} \right]=\emptyset$.

Let $w=(w^1,\ldots,w^n) \in E.$ Then
\begin{equation} 
\label{vetor w}
|w^i|<\varepsilon \text{ for } i=1,\ldots,n-1 \text{ and } w^n \in \left( \frac{\Vert v\Vert}{2},2 \Vert v\Vert \right). 
\end{equation}
A supporting hyperplane $H_w$ of $S_F[0,F(w)]$ at $w$ separates ${\mathbb R^n}$ in two parts and $B_F(0,F(w))$ lies completely in the same part of the origin. 
Moreover $D \subset B_F(0,F(w))$ due to (\ref{desigualdades}). 
This implies that $H_w$ does not intercept $D$. 
In particular $(0,\ldots,0,1)$ is not parallel to  $H_w$ because $\tilde\pi(w) \in D$. Therefore $H_w$ intercepts the $y^n$-axis. This intersection is in the positive part of $y^n$-axis otherwise the line connecting $w$ and this intersection would intercept $D$ as well.

Denote the equation of $H_w$ by
\[
a^1y^1+ \cdots + a^ny^n=1.
\]
Considering the intersection of $H_w$ with the $y^1, \ldots, y^{n-1}$ axes, we have that
\begin{equation} 
\label{a^i}
|a^i|\leq \frac{2r_M}{r_m \Vert v\Vert} \ \ {\mbox{for}} \ \ i=1,\ldots,n-1
\end{equation}
because $H_w$ does not intercept $D.$ In order to estimate $a^n$ notice that
\[
a^n=\frac{1-a^1w^1-\cdots-a^{n-1}w^{n-1}}{w^n} \geq \frac{1-|a^1w^1|-\cdots-|a^{n-1}w^{n-1}|}{w^n},
\]
where $|a^iw^i|\leq \frac{1}{2n}$ due to (\ref{vetor w}) and (\ref{a^i}). Therefore
\begin{equation} 
\label{a^n}
a^n > \frac{1}{4\Vert v\Vert}
\end{equation}
because $w^n < 2\Vert v\Vert$.

The linear function $L_w$ such that $L_w=F(w)$ on $H_w$ is given by
\[
L_w(y)=F(w)(a^1y^1+a^2y^2+\cdots+ a^ny^n).
\]
The directional derivative of $L_w$ with respect to $v/\|v\|=(0,\ldots,0,1)$ satisfies
\begin{equation} 
\label{derivada da L_w}
\langle\nabla L_w(y),(0,\ldots,0,1)\rangle = a^nF(w)> \frac{1}{4 \Vert v\Vert}F(w) > \frac{r_m}{8}
\end{equation} 
due to (\ref{desigualdades}) and (\ref{a^n}).
Finally the inequality $F \geq L_w$ together with (\ref{derivada da L_w}) implies that
\begin{eqnarray*}
F\left(w+t\frac{v}{\|v\|}\right)-F(w) & \geq & L_w\left(w+t\frac{v}{\|v\|}\right)-L_w(w) \\ 
= t \langle\nabla L_w(y),(0,\ldots,0,1)\rangle & 
> & \frac{r_m}{8}t.\square
\end{eqnarray*}

\section{Vertical Smoothing of $C^0$-Finsler Structures}
\label{section vertical smoothing}

Now we will define the vertical smoothing of a $C^0$-Finsler structure. This process will be done in a coordinate system.

Let $(M,F)$ be a $C^0$-Finsler manifold and consider a coordinate system $(x^1,\ldots,x^n)$ on an open subset $U_0$ of $M.$ Choose an open subset $U$ of $U_0$ with compact closure such that
\begin{equation} 
\label{definindo U}
\overline{U} \subset U_0.
\end{equation} 
We consider $U$ with the coordinate system $(x^1,\ldots,x^n).$ Let
\begin{equation} 
\label{sistema de coordenadas x,y}
(x,y) = (x^1,\ldots,x^n,y^1,\ldots,y^n)
\end{equation}
be the natural coordinates on $TU$ with respect to $(x^1,\ldots,x^n).$

As in Section \ref{section growth rates}, we will work with two asymmetric norms in each tangent space $T_xU:$ The Euclidean norm  with respect to the coordinate system $(y^1,\ldots,y^n)$ and $F$. 
For instance the closed ball centered at $v$ and radius $r$ in $T_xU $ with respect to norm $F$ is denoted by $B_F[x,v,r]$ and so on.

Analogously to Section \ref{section growth rates}, define
\begin{equation} 
\label{constante r_U}
r_U= \max_{(x,y) \in \ \overline{U}\times S_{||\cdot||}[x,0,1]} F(x,y) = \max_{(x,y) \in \ \overline{U}\times B_{||\cdot||}[x,0,1]} F(x,y)
\end{equation}
and
\begin{equation} 
\label{constante rho_U}
r_u= \min_{(x,y) \in \ \overline{U}\times S_{||\cdot||}[x,0,1]} F(x,y).
\end{equation}
The following version of Lemma \ref{proposicao F em Rn} is straightforward:

\begin{lem} 
\label{lema F em TU}
Let $F$ be a $C^0$-Finsler structure on $M,$ $U$ be an open subset satisfying (\ref{definindo U}) and  $(x^1,\ldots,x^n,y^1,\ldots,y^n)$ be a coordinate system on $TU$ as defined in (\ref{sistema de coordenadas x,y}). Let $v\in {\mathbb R^n}$ be a nonzero vector and consider
\begin{equation} 
\label{intervalo contendo epsilon}
\varepsilon \in \left(0,\frac{r_u \Vert v \Vert}{4nr_U}\right],
\end{equation}
where $r_U$ and $r_u$ are defined by (\ref{constante r_U}) and (\ref{constante rho_U}) respectively. Then
\[
F\left(x,w+t\frac{v}{\|v\|}\right) - F(x,w) > \frac{r_u}{8}t
\]
for every $x \in U,$ $w \in B_{\|\cdot\|}[x,v,\varepsilon]$ and $t>0$.
\end{lem}

\begin{proof}
Analogous to the proof of Lemma \ref{proposicao F em Rn}.
\end{proof}

The first trial for the vertical smoothing in $U$ would be defining
\[
(\eta_\varepsilon \ast_v F) (x,y) = \int \eta_\varepsilon(z)F(x,y-z)dz
\]
for $\varepsilon $ given by (\ref{intervalo contendo epsilon}), where $\ast_v := \ast_2$ stands for the mollifier smoothing on the vertical part (along the fibers of $TM$). 
We can prove that for every  $x \in U$, $T_xU \cap (\eta_\varepsilon \ast_v F)^{-1}(r_U)$ is a smooth hypersurface in $T_xU$ such that its radial projection onto  $S_{\|\cdot\|}[x,0,1]$ is a diffeomorphism. 
Then we could define the vertical smoothing $G_\varepsilon$ of $F$ as the $C^0$-Finsler structure such that $S_{G_\varepsilon}[x,0,r_U]$ is given by $T_xU \cap (\eta_\varepsilon \ast_v F)^{-1} (r_U)$. 
We can prove that everything works nicely except for the fact that these asymmetric norms aren't necessarily strongly convex. This problem is solved using
\[
\zeta_\varepsilon = (1-u(\varepsilon))\eta_\varepsilon+u(\varepsilon ) \eta_{\frac{2r_U}{r_u}}
\] 
instead of $\eta_\varepsilon$, where 
\begin{equation}
\label{define-u}
u:(0,1) \rightarrow (0,\infty)
\end{equation} 
is a increasing function satisfying $\lim_{\varepsilon \rightarrow 0}u(\varepsilon) = 0$. The ``perturbation'' due to $u(\varepsilon ) \eta_{\frac{2r_U}{r_u}}$ is enough to assure the strong convexity of $\zeta_\varepsilon$ due to Lemma \ref{condicao para curvatura}.  
Of course there are some work to do, like redefining $\varepsilon$ in order to make everything work. We are going to make these calculations in sequel.

First of all we define the interval of variation $(0,\tau)$ of $\varepsilon$ choosing
\begin{equation}
\label{define-tau}
\tau  \in \left( 0, \frac{r_u}{16nr_U} \right)
\end{equation}
such that
\[
u(\varepsilon) \in \left( 0, \frac{r_u}{16nr_U} \right)
\]
for every $\varepsilon \in (0,\tau)$. For a fixed $\varepsilon$ define
\[
(\zeta_\varepsilon \ast_v F) (x,y) = \int \zeta_\varepsilon(z)F(x,y-z)dz.
\]
It is straightforward that $(\zeta_\varepsilon \ast_v F)$ is continuous and $(\zeta_\varepsilon \ast_v F)(x,\cdot)$ is smooth for every $x\in U.$ We use the notation $\Hessian_y(\zeta_\varepsilon \ast_v F)(x,y)$ in order to denote the Hessian of $(\zeta_\varepsilon \ast_v F)$ with respect to $(y^1,\ldots,y^n)$.

Notice that
\[
\max_{(x,y) \in \ \overline{U}\times S_{\|\cdot\|} \left[ x,0,\frac{1}{2} \right]}F(x,y) = \frac{r_U}{2},
\]
\[
\min_{(x,y) \in \ \overline{U}\times S_{\|\cdot\|}\left[ x,0,\frac{2r_U}{r_u} \right]}F(x,y) = 2r_U
\]
and
\[
\mathcal A := \left\{(x,y) \in TU; x \in U, \|y\| \in \left[ \frac{1}{2}, \frac{2r_U}{r_u} \right] \right\}
\] 
contains $F^{-1}([r_U/2,2r_U])$.
Define
\[
\mathcal A_{\frac{1}{2}} := \left\{(x,y) \in \mathcal A; \|y\| = \frac{1}{2}\right\}
\]
and
\[
\mathcal A_{\frac{2r_U}{r_u}} := \left\{(x,y) \in \mathcal A; \|y\|=\frac{2r_U}{r_u} \right\}.
\]

The following lemma will be used to construct the spheres in each tangent space of the vertical mollifier smoothing.

\begin{lem} 
\label{lema propriedade de F_epsilon} 
$(\zeta_\varepsilon \ast_v F)$ has the following properties for every $\varepsilon \in \left(0, \tau\right):$
\begin{enumerate}
\item $\Hessian_y(\zeta_\varepsilon \ast_v F)(x,y)$ is positive definite for every $(x,y)\in \mathrm{int} \mathcal A$;
\item The \em{vertical radial derivative} of $(\zeta_\varepsilon \ast_v F)$ on $\mathcal A$ is strictly positive. That is,
\[
y^i(\zeta_\varepsilon \ast_v F)_{y^i}(x,y) > 0, \forall (x,y) \in \mathcal A;
\]
\item $(\zeta_\varepsilon \ast_v F)(x,y)<r_U$ for every $(x,y) \in TU$ with $\|y\| \leq \frac{1}{2};$
\item $(\zeta_\varepsilon \ast_v F)(x,y)>r_U$ for every $(x,y) \in TU$ with $\|y\| \geq \frac{2r_U}{r_u}.$
\end{enumerate}
In particular, for every $x \in U$, we have that $(\zeta_\varepsilon \ast_v F)^{-1}(r_U)\cap T_xU$ is a smooth hypersurface in $T_xU$ such that its radial projection onto  $S_{\|\cdot\|}[x,0,1]$ is a diffeomorphism.
\end{lem}

\begin{proof}

\

(1) Given $(x,y) \in \mathrm{int}\mathcal A$ and $\xi\in {\mathbb R^n} \setminus\{0\},$ we must prove that
\[
(\zeta_\varepsilon \ast_v F)_{y^iy^j}(x,y)\xi^i\xi^j>0.
\]
After an orthogonal change of coordinates $(y^1, \ldots, y^n) \mapsto (z^1, \ldots, z^n)$ that fixes $(0,\ldots, 0)$ and satisfies $\frac{\partial}{\partial z^n} = \xi$, we have that
\[
(\zeta_\varepsilon \ast_v F)_{y^iy^j}(x,y)\xi^i\xi^j = (\zeta_\varepsilon \ast_v F)_{z^nz^n}(x,z).
\]
We must prove that the latter is strictly positive.
Notice that
\[
(\zeta_\varepsilon \ast_v F)_{z^nz^n} (x,z) = (1- u( \varepsilon)) (\eta_\varepsilon \ast_v F)_{z^nz^n} (x,z) + u(\varepsilon ) (\eta_{\frac{2r_U}{r_u}} \ast_v F)_{z^nz^n}(x,z).
\]
By Lemma \ref{condicao para curvatura}, we have that $(\eta_\varepsilon \ast_v F)_{z^nz^n}(x,z)\geq 0$. 
Now we use Fubini's Theorem in order to prove that  $ (\eta_{\frac{2r_U}{r_u}} \ast_v F)_{z^nz^n}(x,z)>0$. 
Write $z = (z^\prime, z^n)$, where $z^\prime = (z^1, \ldots, z^{n-1})$.
We have that
\begin{eqnarray*}
& & \left(\eta_{\frac{2r_U}{r_u}} \ast_v F\right)_{z^nz^n} (x,z)
= \int_{{\mathbb R^n}} \left( \eta_{\frac{2r_U}{r_u}}\right)_{z^nz^n}(z-w)F(x,w) dw \\
& = & \int_{{\mathbb R}^{n-1}} \left( \int_{\mathbb R} \left( \eta_{\frac{2r_U}{r_u}}\right)_{z^zz^n}(z^\prime - w^\prime, z^n - w^n)F(x,w',w^n) dw^n\right) dw'.
\end{eqnarray*}
Consider 
\[
\eta_{\frac{2r_U}{r_u}}(z^\prime - w^\prime, \cdot): \mathbb R \rightarrow {\mathbb R}.
\]
Depending on $w^\prime$, the function $\eta_{\frac{2r_U}{r_u}}(z'-w', \cdot)$ is identically zero or else its graph has the bell shape described in Remark \ref{formato eta}. 
These facts are direct consequences of the definition of $\eta$.
In either case we have that
\[
\int_{\mathbb R} \left( \eta_{\frac{2r_U}{r_u}} \right)_{z^nz^n} (z^\prime - w^\prime, z^n - w^n)F(x,w^\prime,w^n) dw^n  \geq   0
\]
due to Lemma \ref{condicao para curvatura}.
In the particular case where $w^\prime = 0$, 
$F(x,0,\cdot)$ is an asymmetric norm on $\mathbb R$ and $\eta_{\frac{2r_U}{r_u}}(z^\prime, z^n-w^n) > 0$ when $w^n=0$ because $\Vert z\Vert < \frac{2r_U}{r_u}$. 
The function $w^n \mapsto\eta_{\frac{2r_U}{r_u}}(z^\prime, z^n - w^n)$ has a bell shape and $w^n=0$ lies in its support. 
Therefore $F(x,0, \cdot)$ is convex and its restriction to the support of $w^n \mapsto\eta_{\frac{2r_U}{r_u}}(z^\prime, z^n - w^n)$ is not affine.
Hence 
\[
\int_{z^n-\frac{2r_U}{r_u}}^{z^n+\frac{2r_U}{r_u}} \left( \eta_{\frac{2r_U}{r_u}} \right)_{z^nz^n} (z^\prime, z^n-w^n)F(x, 0, w^n) dw^n  >  0
\]
due to Lemma \ref{condicao para curvatura}, what settles Item (1).

\

(2) Consider $(x,y) \in \mathcal A$. 
Then
\begin{eqnarray*}
& & (\zeta_\varepsilon \ast_v F) \left(x,y+t\frac{y}{\|y\|}\right)-(\zeta_\varepsilon \ast_v F) (x,y) \\
& = & (1- u(\varepsilon ))\int_{B_{\|\cdot\|}[x,y,\varepsilon]} \eta_\varepsilon(y-z)\left[ F\left(x,z+t\frac{y}{\|y\|}\right) - F(x,z) \right]dz\\
& + & u( \varepsilon ) \int_{B_{\|\cdot\|}[x,0,(2r_U)/r_u]} \eta_{\frac{2r_U}{r_u}}(z) \left[ F\left( x,y-z+t\frac{y}{\|y\|}\right) - F(x,y-z)\right] dz \\
& \geq & \frac{(1-u(\varepsilon)) r_u}{8}t-u(\varepsilon) t \int_{B_{\|\cdot\|}[x,0,(2r_U)/r_u]} \eta_{\frac{2r_U}{r_u}}(z)F\left(x,\frac{y}{\|y\|}\right) dz \\
& \geq & \frac{(1-u(\varepsilon))r_u}{8}t - u(\varepsilon) t r_U
 > \frac{15}{16}\frac{r_u}{8}t-\frac{r_u}{16r_U}tr_U 
 = \frac{7}{128}r_u t,
\end{eqnarray*}
for $t>0$ due to Lemma \ref{lema F em TU}.  Thus $\frac{7}{128}r_u$ is a lower bound for $\frac{y^i}{||y||}(\zeta_\varepsilon \ast_v F)_{y^i}(x,y)$.

\

(3) Consider $(x,y) \in TU$ such that $\Vert y \Vert \leq 1/2$. 
Then
\begin{eqnarray*}
(\zeta_\varepsilon \ast_v F)(x, {y}) & = & \int \left[(1-u(\varepsilon))\eta_\varepsilon(z) + u(\varepsilon)\eta_{\frac{2r_U}{r_u}}(z) \right] F(x, {y}-z)dz \\
& \leq & \int \left[(1-u(\varepsilon))\eta_\varepsilon(z)+u(\varepsilon)\eta_{\frac{2r_U}{r_u}}(z) \right] \left[ F(x, {y})+F(x,z)\right]dz \\
& \leq & \int_{B_{\|\cdot\|}[x,0,\varepsilon]} \eta_\varepsilon(z) \left[ F(x, {y})+F(x,z)\right]dz \\
& + & u(\varepsilon)\int_{B_{\|\cdot\|}[x,0,(2r_U)/r_u]} \eta_{\frac{2r_U}{r_u}}(z) \left[ F(x, {y})+F(x,z)\right]dz \\ 
& \leq & \frac{r_U}{2} + \frac{r_u}{16nr_U}r_U+\frac{r_u}{16nr_U}\frac{r_U}{2}+\frac{r_u}
{16nr_U}\frac{2r_U}{r_u}r_U
\leq \frac{23}{32}r_U < r_U.
\end{eqnarray*}

(4) Consider $(x,y)\in TU$ such that $\Vert y\Vert \geq \frac{2r_U}{r_u}$. Then
\begin{eqnarray*}
(\zeta_\varepsilon \ast_v F)(x, {y})& = & \int \left[(1-u(\varepsilon))\eta_\varepsilon(z)+u(\varepsilon)\eta_{\frac{2r_U}{r_u}}(z) \right] F(x, {y}-z)dz \\
& \geq & \int (1-u(\varepsilon))\eta_\varepsilon(z) F(x, {y}-z)dz \\
& \geq & \int (1-u(\varepsilon))\eta_\varepsilon(z) F(x, {y})dz-\int (1-u(\varepsilon))\eta_\varepsilon(z) F(x,z)dz \\
& \geq & (1-u(\varepsilon))2r_U-(1-u(\varepsilon))\frac{r_u}{16} \\
& \geq & 2\frac{15}{16}r_U-\frac{r_U}{16} 
= \frac{29}{16} r_U > r_U
\end{eqnarray*}
what settles the lemma.
\end{proof}

The next lemma is an intermediate step in order to prove several uniform convergences afterwards. 

\begin{lem} 
\label{convergencia uniforme de F til} 
$(\zeta_\varepsilon \ast_v F) \rightarrow F$ uniformly on compact subsets of  $TU\cong U \times \mathbb R^n$. 
\end{lem}

\begin{proof}
It is a direct consequence of Proposition \ref{funcao convergencia uniforme}.
\end{proof}

Let $x \in U$ and $\varepsilon \in \left(0, \tau \right)$. 
Lemma \ref{lema propriedade de F_epsilon} states that $(\zeta_\varepsilon \ast_v F)^{-1}(r_U)\cap T_xU $ is a smooth hypersurface in $T_xU$ such that its radial projection onto the Euclidean sphere in $T_xU$ is a diffeomorphism. 
Moreover $(\zeta_\varepsilon \ast_v F)(x, \cdot)$ is a convex function due to Lemma \ref{suavizacao mollifier e convexa} and $(\zeta_\varepsilon \ast_v F)^{-1}([0, r_U])\cap T_xU $ is a convex subset containing the origin.
We define the function 
\begin{equation} \label{aplicacao G_epsilon}
G_\varepsilon: TU \rightarrow {\mathbb R}
\end{equation}
such that $G_\varepsilon(x,\cdot): T_xU \rightarrow \mathbb R$ is an asymmetric norm with 
\[
S_{G_\varepsilon(x, \cdot)}[0,r_U]=(\zeta_\varepsilon \ast_v F)^{-1}(r_U)\cap T_xU.
\]
We are going to prove that $G_\varepsilon$ is a $C^0$-Finsler structure on $TU$ (see Proposition \ref{propriedades da aplicacao G epsilon}).
Consider a coordinate system $(x^i)$ on $U$ and $((x^i),(y^i))$ the respective natural coordinate on $TU$. 
We have the identification $TU\cong (x^i)(U) \times \mathbb R^n$ and we define spherical coordinates in the vertical part $\mathbb R^n$ of $TU$. 
Let $(\theta^i) = (\theta^2,\ldots,\theta^n):S_\theta \rightarrow W_\theta \subset \mathbb R^{n-1}$ be a coordinate system on an open subset $S_\theta$ of the Euclidean sphere $S_{\Vert \cdot\Vert}[0,1]$. 
If $r>0$ is the radial coordinate in ${\mathbb R^n},$ then $(r,\theta):=(r,\theta^2,\ldots,\theta^n)$ is a coordinate system in the open cone $C(S_\theta)\subset \mathbb R^n$ of $S_\theta$ with the vertex at the origin. 
Varying $x \in U$, we have that
\begin{equation} 
\label{sistema induzido}
((x^i),r,(\theta^i)):U\times C(S_\theta) \rightarrow (x^i)(U)\times(0,\infty) \times W_\theta , 
\end{equation}
is a coordinate system on $TU\backslash 0$.

Let $(x_0,y_0)\in TU$ with $y_0\in S_{G_\varepsilon(x_0, \cdot)}[0,r_U]$.
Consider a coordinate system $((x^i),r,(\theta^i))$ as defined in (\ref{sistema induzido}) in a neighborhood of $(x_0,y_0)$. 
Let
\[
(x_0, r_0, \theta_0) := (x_0^1, \ldots, x_0^n, r_0,\theta_0^2,\ldots,\theta_0^{n })
\]
be the coordinates of $(x_0, y_0)$.
Item (2) of  Lemma \ref{lema propriedade de F_epsilon} implies that 
\[
\frac{\partial (\zeta_\varepsilon \ast_v F)}{\partial r}(x_0,r_0,\theta_0) > 0
\] 
and the implicit function theorem states that there exist a smooth function
\[
\phi_\varepsilon(x_0,\cdot): W_\theta \rightarrow \left(\frac{1}{2},\frac{2r_U}{r_u}\right)
\]
satisfying the following conditions:
\begin{enumerate}[(i)]
\item For every $\theta \in W_\theta$, $\phi_\varepsilon(x_0,\theta)$ is the unique value such that 
\[
(\zeta_\varepsilon \ast_v F)(x_0,\phi_\varepsilon(x_0,\theta),\theta)=r_U;
\]
\item
\[
\frac{\partial \phi_\varepsilon(x_0,\theta)}{\partial \theta^i} 
= - \frac{\frac{\partial (\zeta_\varepsilon \ast_v F)(x_0,\phi_\varepsilon(x_0,\theta),\theta) }{\partial \theta^i}}{\frac{\partial (\zeta_\varepsilon \ast_v F)(x_0,\phi_\varepsilon(x_0,\theta),\theta) }{\partial r}}
\]
for every $\theta \in W_\theta$.
\end{enumerate}
Varying $x $ in $U,$ we have the function
\begin{equation} 
\label{aplicacao phi_epsilon}
\phi_\varepsilon:U\times W_\theta \rightarrow \left(\frac{1}{2},\frac{2r_U}{r_u}\right)
\end{equation}
such that $(\zeta_\varepsilon \ast_v F)(x, \phi_\varepsilon(x,\theta),\theta)=r_U$ for every $ (x,\theta) \in U\times W_\theta$.
Therefore the expression of $G_\varepsilon$ in terms of $(x,r,\theta) \in U\times (0,\infty) \times W_\theta$ is given by
\begin{equation} 
\label{G_epsilon em coordenadas esfericas} 
G_\varepsilon (x,r,\theta) = \frac{r .r_U}{\phi_\varepsilon(x, \theta)}
\end{equation}
due to the positive homogeneity of $G_\varepsilon$ with respect to the variable $r$.

\begin{pro} 
\label{propriedades da aplicacao G epsilon}  
The function $G_\varepsilon$ defined in (\ref{aplicacao G_epsilon}) satisfies the following properties for every $\varepsilon \in \left(0,\frac{r_u}{16nr_U}\right):$
\begin{enumerate}
\item $G_\varepsilon$ is a $C^0$-Finsler structure; 
\item $G_\varepsilon(x,\cdot)$ is smooth in $T_xU \setminus \{0\};$
\item $\frac{\partial^{|\alpha|} G_\varepsilon}{\partial (y^1)^{\alpha_1} \cdots \partial ({y^n})^{\alpha_n}}$ is continuous in $U\times ({\mathbb R^n}\setminus \{0\})$ for each multiindex $\alpha=(\alpha_1,\ldots,\alpha_n)$.
\end{enumerate}
\end{pro}

\begin{proof}

\

(1) We need to prove that $G_\varepsilon$ is continuous in order to settle Item (1).
In order to prove that $G_\varepsilon$ is continuous in $TU\backslash 0$, we will prove that $\phi_\varepsilon$ defined in (\ref{aplicacao phi_epsilon}) is continuous.

Let $U$ and $W_\theta$ as in the definition of (\ref{aplicacao phi_epsilon}). 
Let $\delta>0$, $(x_0, \theta_0)\in U\times W_\theta$  and $r_0 = \phi_\varepsilon(x_0,\theta_0)$. 
Without loss of generality, we can suppose that $r_0-\delta$, $r_0+\delta$ $\in (1/2,(2r_U)/r_u).$ 
We must find a neighborhood $Z$ of $(x_0,\theta_0)$ such that
\begin{equation} 
\label{phi epsilon}
r_0-\delta < \phi_\varepsilon(x,\theta)  < r_0 +\delta
\end{equation}
for every $(x,\theta) \in Z$.

We know that $\frac{\partial (\zeta_\varepsilon \ast_v F)}{\partial r}(x,r,\theta) > 0$ for every $(x,r,\theta) \in U \times (1/2,(2r_U)/r_u)$ $\times W_\theta.$ Then the function
\[
r\mapsto (\zeta_\varepsilon \ast_v F)(x_0,r,\theta_0), \  r \in (1/2,(2r_U)/r_u)
\]
is strictly increasing.  
As $(\zeta_\varepsilon \ast_v F)(x_0, r_0, \theta_0)=r_U,$ then 
\[
(\zeta_\varepsilon \ast_v F)(x_0, r_0 - \delta,\theta_0)< r_U
\] 
and there exist neighborhoods $Z_1 \subset U$ of $x_0$ and $Z_2 \subset  W_\theta$ of $\theta_0$ such that 
\[
(\zeta_\varepsilon \ast_v F)(x, r_0-\delta, \theta)< r_U,\hspace{5mm} \forall(x,\theta) \in Z_1\times Z_2.
\]
Therefore given $(x,\theta) \in Z_1\times Z_2$ we have that
\[
r_0-\delta < \phi_\varepsilon(x,\theta),
\]
because $(\zeta_\varepsilon \ast_v F) (x,\phi_\varepsilon(x,\theta),\theta)=r_U.$

Analogously there exists a neighborhood $\tilde{Z_1}\times\tilde{Z_2} \subset U\times W_\theta$ of $(x_0,\theta_0)$ such that
\[
\phi_\varepsilon(x,\theta)< r_0 + \delta, \hspace{5mm} \forall (x,\theta) \in \tilde{Z_1}\times \tilde{Z_2}
\]
and $(x,\theta) \in Z=(Z_1\cap \tilde{Z_1}) \times (Z_2\cap \tilde{Z_2})$ satisfies (\ref{phi epsilon}).
Therefore $G_\varepsilon$ is continuous in $TU\backslash 0$.

In order to prove that $G_\varepsilon$ is continuous at the points $(x,0) \in TU$, notice that for every $(x,y) \in TU\backslash 0$ we have
\[
\frac{G_\varepsilon(x,y)}{\|y\|_{\mathbb R^n}}
=\frac{G_\varepsilon(x,r,\theta)}{r}
=\frac{r_U}{\phi_\varepsilon(x,\theta)} \leq 2r_U
\]
due to (\ref{aplicacao phi_epsilon}) and (\ref{G_epsilon em coordenadas esfericas}).
Therefore
\[
G_\varepsilon(x,y) \leq 2r_U \|y\|_{\mathbb R^n},\hspace{5mm} \forall (x,y) \in TU,
\]
what is enough to prove the continuity at $(x,0) \in TU$.

\

(2) It follows from the fact that $\phi_\varepsilon(x,\cdot)$ is smooth.

\

(3) First of all observe that all partial derivatives of $\zeta_\varepsilon \ast_v F$ with respect to $r,\theta^2, \ldots, \theta^n$ are continuous due to Lemma \ref{derivada zeta vertical continua} and the chain rule.

In order to prove that $(G_\varepsilon)_{y^i}(x,y)$, $i=1, \ldots, n$, are continuous in $U\times ({\mathbb R^n}\setminus\{0\})$, it is enough to prove that $(G_\varepsilon)_r(x,r,\theta) $ and $(G_\varepsilon)_{\theta^j}(x,r,\theta) $, $j=2, \ldots, n$, are continuous in $U\times(0,\infty)\times W_\theta$ due to the chain rule. 

We know from the proof of Item (1) that $\phi_\varepsilon(x,\theta)$ is continuous and 
\begin{equation}
\label{g epsilon r}
(G_\varepsilon)_r(x,r,\theta) = \frac{r_U}{\phi_\varepsilon(x,\theta)}
\end{equation}
is continuous in $U\times(0,\infty)\times W_\theta.$ 
In addition
\begin{equation}
\label{g epsilon theta}
(G_\varepsilon)_{\theta^j}(x,r,\theta) 
= - \frac{rr_U}{(\phi_\varepsilon(x,\theta))^2}\frac{\partial \phi_\varepsilon(x,\theta)}{\partial \theta^j},
\end{equation}
where
\begin{equation}
\label{phi epsilon theta j}
\frac{\partial \phi_\varepsilon(x,\theta)}{\partial \theta^j} 
= - \frac{\frac{\partial (\zeta_\varepsilon \ast_v F)(x,\phi_\varepsilon(x,\theta),\theta) }{\partial \theta^j}}{\frac{\partial (\zeta_\varepsilon \ast_v F)(x,\phi_\varepsilon(x,\theta),\theta) }{\partial r}}
\end{equation}
for every $(x,\theta) \in U\times W_\theta$. 
Moreover $\phi_\varepsilon (x, \theta)$ and 
\[
\frac{\partial (\zeta_\varepsilon \ast_v F)(x,\phi_\varepsilon(x,\theta),\theta) }{\partial r}
\]
are different from zero. 
Therefore $(G_\varepsilon)_{\theta^j}(x,r,\theta)$ is continuous in $U \times (0,\infty) \times W_\theta$. 

For the derivatives of higher order, notice that due to the chain rule, all the partial derivatives of $G_\varepsilon$ with respect to $y^1, \ldots, y^n$ are continuous iff all its partial derivatives with respect to $r, \theta^2, \ldots, \theta^n$ are continuous. 
We will prove the continuity of the partial derivatives with respect to the latter variables.

The continuity of the derivatives of $G_\varepsilon$ of order $\vert \alpha \vert$ can be made by induction. 
The continuity of the derivatives of $G_\varepsilon$ of order two follows from (\ref{g epsilon r}), (\ref{g epsilon theta}), (\ref{phi epsilon theta j}) and the fact that they can be written in terms of $\phi_\varepsilon(x,\theta)$ and the derivatives of $(\zeta_\varepsilon \ast_v F)(x,\phi_\varepsilon(x,\theta),\theta)$ of order up to $2$.
Proceeding inductively, we have that the derivatives of $G_\varepsilon$ of order $\vert \alpha \vert$ are given in terms of $\phi_\varepsilon (x,\theta)$ and the derivatives of $(\zeta_\varepsilon \ast_v F)(x,\phi_\varepsilon(x,\theta),\theta)$
of order up to $\vert \alpha\vert$.
\end{proof}

\begin{pro}
\label{g epsilon radialmente convexo}
The function $G_\varepsilon(x,\cdot): T_xU \rightarrow \mathbb R$ is a Minkowski norm for every $x \in U$.
\end{pro}

\begin{proof}
$\Hess_y (\zeta_\varepsilon \ast_v F)(x, \cdot)$ is positive definite in $\mathrm{int} \mathcal A \supset S_{G_\varepsilon}[x,0,r_U]$ due to Lemma \ref{lema propriedade de F_epsilon}.
Proceeding as in (\ref{esfera local}) and (\ref{F e phi no kernel}), with $F^2$ replaced by $\zeta_\varepsilon \ast_v F$ and $r^2$ replaced by $r_U$, we have that $S_{G_\varepsilon}[x,0,r_U]$ can be written locally as graph of a function $\psi: \mathbb R^{n-1} \rightarrow \mathbb R$ such that $\psi(0)=0$, $d\psi_0 = 0$ and $\Hess \psi(0)$ is positive definite.
Therefore $G_\varepsilon(x, \cdot)$ is a Minkowski norm due to Theorem \ref{control minkowski boundary}.
\end{proof}

\begin{pro}
\label{g epsilon converge uniforme}
$G_\varepsilon \rightarrow F $ uniformly on compact subsets of $TU\cong U \times \mathbb R^n$.
\end{pro}

\begin{proof}
Let $K_{TU}$ be a compact subset of $TU$.
Let $\tilde r > 0 $ be a constant  such that $\|y\|\leq \tilde r$ for every $(x,y) \in K_{TU}$. 
Consider $\delta>0$. 
Suppose that $K_{TU}$ is contained in a coordinate neighborhood as defined in  (\ref{sistema induzido}) (The general case follows by covering $K_{TU}$ by a finite number of such coordinate neighborhoods).
By Lemma \ref{convergencia uniforme de F til} there exists $\mu > 0$ such that if $\varepsilon \in (0,\mu),$ then
\begin{equation}
\label{converge uniforme proposicao}
|(\zeta_\varepsilon \ast_v F)(x, r,\theta)-F(x, r,\theta)|<  \frac{\delta}{2\tilde r}
\end{equation}
for every $(x, r,\theta) \in K_{TU}$.  
Then
\begin{eqnarray*}  
|G_\varepsilon (x,r,\theta)-F (x,r,\theta)| 
& = & \left|G_\varepsilon \left(x,r\frac{\phi_\varepsilon(x,\theta)}{\phi_\varepsilon(x,\theta)},\theta\right)
- F \left(x,r\frac{\phi_\varepsilon(x,\theta)}{\phi_\varepsilon(x,\theta)},\theta \right) \right| \\
& = & r  \frac{ \left|G_\varepsilon \left(x, {\phi_\varepsilon(x,\theta)},\theta \right)- F \left(x, {\phi_\varepsilon(x,\theta)},\theta\right) \right|}{\phi_\varepsilon(x,\theta)} \\
& = & r  \frac{ \left|(\zeta_\varepsilon \ast_v F) \left(x, {\phi_\varepsilon(x,\theta)},\theta \right)- F \left(x, {\phi_\varepsilon(x,\theta)},\theta \right) \right|}{\phi_\varepsilon(x,\theta )} \\
& < & \frac{r}{\phi_\varepsilon(x,\theta)} \frac{\delta}{2\tilde r} 
< \delta
\end{eqnarray*}
for every $\varepsilon \in (0,\mu)$ and $(x,r,\theta) \in K_{TU}$ due to (\ref{converge uniforme proposicao}) and the fact that $\phi_\varepsilon(x,\theta)\in (1/2,(2r_U)/r_u)$. 
Thus $G_\varepsilon \rightarrow F$ uniformly on $K_{TU}$.
\end{proof}

\section{The mollifier smoothing of $F$}
\label{secao mollifier smoothing F}

In this section we define the mollifier smoothing $F_\varepsilon:TM\rightarrow {\mathbb R}$ of $F$ and we study its uniform convergence to $F$ as $\varepsilon$ goes to zero.

Let 
\begin{equation}
\label{define U lambda}
\{(U_\lambda,(x^i)_\lambda)\}_{\lambda\in\Lambda}
\end{equation} 
be a locally finite differentiable structure on $M$  such that $U_\lambda$ is defined as $U$ in (\ref{definindo U}). 
Let $\{\varphi_\lambda\}_{\lambda\in \Lambda}$ be a differentiable partition of unity on $M$ subordinate to the covering  $\{U_\lambda\}_{\lambda\in\Lambda}$ of $M.$ For each $\lambda \in \Lambda,$ denote $V_\lambda=\mathrm{int}(\supp(\varphi_\lambda))$ and $\nu_\lambda=\dist_{{\mathbb R^n}}(\partial V_\lambda,\partial U_\lambda).$

Fix $\lambda \in \Lambda$.  Let $r_{U_\lambda}$, $r_{u_\lambda}$, $u_\lambda$ and $\tau_\lambda$ as defined in (\ref{constante r_U}), (\ref{constante rho_U}), (\ref{define-u}) and (\ref{define-tau}) respectively. Let $\delta_\lambda = \min\{\tau_\lambda,\nu_\lambda\}.$ For every $\varepsilon \in (0,\delta_\lambda)$  consider the function
\[
G_{\varepsilon,\lambda}:TU_\lambda \rightarrow {\mathbb R}
\] 
defined as in (\ref{aplicacao G_epsilon}).
$G_{\varepsilon,\lambda}$ is continuous due to Proposition \ref{propriedades da aplicacao G epsilon} and we can define the function $\check F_{\varepsilon,\lambda}$ by
\[
\check F^2_{\varepsilon,\lambda}(x,y)=  \int \eta_{\varepsilon}(x-z) G^2_{\varepsilon,\lambda}(z,y) dz, \hspace{5mm}\forall (x,y) \in TO_{\varepsilon,\lambda},
\]
where $O_{\varepsilon, \lambda}$ is an open subset of $U_\lambda$ containing $\overline{V_\lambda}$ such that $\check F^2_{\varepsilon, \lambda}$ is well defined.
In order to make the variation of $\varepsilon$ independent of $\lambda,$ define $\psi_\lambda:(0,1) \rightarrow (0,\delta_\lambda)$ by $\psi_\lambda(\varepsilon)=\varepsilon\delta_\lambda$ and set
\begin{equation} 
\label{F check lambda}
F^2_{\varepsilon,\lambda}(x,y)=  \int \eta_{\psi_\lambda(\varepsilon )}(x-z) G^2_{\psi_\lambda(\varepsilon),\lambda}(z,y) dz, 
\end{equation}
where $\varepsilon \in (0,1)$ and $(x,y) \in TO_{\varepsilon, \lambda}$.
Thus, for each $\varepsilon\in (0,1),$ we can define the function $F_\varepsilon:TM\rightarrow [0,\infty)$ by
\begin{equation} 
\label{F epsilon definida em TM}
F^2_\varepsilon(x,y) = \sum_{\lambda \in \Lambda} \varphi_\lambda(x) F^2_{\varepsilon,\lambda}(x,y).
\end{equation}

\begin{defi}
The one parameter family of applications $F_\varepsilon$ given by (\ref{F epsilon definida em TM}) is a {\em mollifier smoothing} of $F$.
\end{defi}

\begin{lem} 
\label{propriedades da aplicacao f epsilon}

\

\begin{enumerate}
\item $F_{\varepsilon,\lambda}$ is a Finsler structure on $O_{\varepsilon,\lambda};$
\item $F_{\varepsilon,\lambda} \rightarrow F$ uniformly on compact subsets $K_{\overline{TV_\lambda}}$ of $\overline{TV_\lambda}$.  
\end{enumerate}
\end{lem}

\begin{proof}

\

(1) $F_{\varepsilon, \lambda}$ is smooth on $TO_{\varepsilon, \lambda} \backslash 0$ due to (\ref{F check lambda}), Proposition \ref{propriedades da aplicacao G epsilon} and the fact that $\eta_{\psi_\lambda(\varepsilon)}$ is smooth.
The positive homogeneity of $ F_{\varepsilon,\lambda}(x, \cdot)$ follows from the positive homogeneity of $G_{\psi_\lambda(\varepsilon),\lambda}(x,\cdot)$.
The positive definiteness of $\Hess_y F_{\varepsilon, \lambda}^2$ is direct consequence of Proposition \ref{g epsilon radialmente convexo} and (\ref{F check lambda}).

\

(2) It is enough to prove that given $\delta>0$, there exist $\beta>0$ such that if $\varepsilon  \in (0,\beta),$ then
\begin{equation}
|F^2_{\varepsilon,\lambda}(x,y)-F^2(x,y)|<\delta \label{menor que delta}
\end{equation}
for every $(x,y) \in K_{\overline{TV_\lambda}}$.
In fact, in this case, there exist $c>0$ and $\beta^\prime \in (0,1)$ such that 
\[
F^2_{\varepsilon,\lambda}(x,y), F^2(x,y) < c
\] 
for every $(x,y) \in K_{\overline{TV_\lambda}}$ and $\varepsilon \in (0,\beta^\prime)$ (This is also true for $\varepsilon \in (0,1)$ but this fact isn't necessary here).
But the square root function is uniformly continuous on $[0,c]$.
Consequently for every $\mu >0$, there exist a $\delta >0$ such that
\begin{equation}
\label{menor que mu}
|F_{\varepsilon,\lambda}(x,y)-F(x,y)| < \mu
\end{equation}
whenever $|F^2_{\varepsilon,\lambda}(x,y)-F^2(x,y)|<\delta$. 
For this $\delta$, we can choose $\beta \in (0,\beta^\prime)$ such that (\ref{menor que delta}) is satisfied whenever $\varepsilon \in (0,\beta)$.
Therefore, for every $\mu >0$, (\ref{menor que mu}) is satisfied whenever $\varepsilon \in (0,\beta)$ and $(x,y) \in K_{\overline{TV_\lambda}}$.

Now let us prove (\ref{menor que delta}).
Set $(\eta_{\psi_\lambda(\varepsilon)} \ast_1 F^2)(x,y)= \int  \eta_{\psi_\lambda(\varepsilon)}(x-z) F^2(z,y)dz$.  
By Proposition \ref{funcao convergencia uniforme}, there exists $\beta_1>0$ such that if $\varepsilon \in (0,\beta_1),$ then
\begin{equation} 
\label{F check}
|(\eta_{\psi_\lambda(\varepsilon)} \ast_1 F^2)(x,y) - F^2(x,y)|< \frac{\delta}{2}
\end{equation}
for every $(x,y) \in K_{\overline{TV_\lambda}}$. Moreover, as a consequence of Proposition \ref{g epsilon converge uniforme}, there exists $\beta_2>0$ such that if $\varepsilon \in (0,\beta_2),$ then
\begin{equation} 
\label{G epsilon uniforme para F}
|G^2_{\psi_\lambda(\varepsilon),\lambda}(z,y) - F^2(z,y)|< \frac{\delta}{2},
\end{equation}
for every $(z,y) \in K_{\overline{TV_\lambda}}$.
Set $\beta = \min\{\beta_1,\beta_2\}$. 
Given $\varepsilon \in (0,\beta)$  we obtain 
\begin{eqnarray}
& & |F^2_{\varepsilon,\lambda}(x,y) - F^2(x,y)| \label{demonstracao F2 converge uniforme} \\
& \leq & |F^2_{\varepsilon,\lambda}(x,y)-(\eta_{\psi_\lambda(\varepsilon)} \ast_1 F^2)(x,y)|+|(\eta_{\psi_\lambda(\varepsilon)} \ast_1 F^2)(x,y) - F^2(x,y)| \nonumber \\
& < & \left|\int \eta_{\psi_\lambda(\varepsilon)}(x-z) [G^2_{\psi_\lambda(\varepsilon),\lambda}(z,y)-F^2(z,y)] dz\right| + \frac{\delta}{2} \nonumber \\ 
& < & \frac{\delta}{2} \int \eta_{\psi_\lambda(\varepsilon) }(x-z) dz + \frac{\delta}{2} = \delta \nonumber
\end{eqnarray}
for every $(x,y) \in K_{\overline{TV_\lambda}}$
due to (\ref{F check}) and (\ref{G epsilon uniforme para F}).
Therefore $F_{\varepsilon,\lambda}(x,y)$ converges uniformly to $F(x,y)$ on $K_{\overline{TV_\lambda}}$.
\end{proof} 

\begin{teo} 
\label{propriedades da aplicacao F epsilon definida em TM} 
A mollifier smoothing $F_\varepsilon: TM \rightarrow \mathbb R$ of a $C^0$-Finsler structure $F$ satisfies the following properties.
\begin{enumerate}
\item $ {F}_{\varepsilon }$ is a Finsler structure on $M$ for every $\varepsilon \in (0,1)$;
\item $ {F}_{\varepsilon } \rightarrow F$ uniformly on compact subsets of $TM$ as $\varepsilon$ converges to zero.
\end{enumerate}
\end{teo}

\begin{proof}

\
 
(1) The functions $\varphi_\lambda:M \rightarrow {\mathbb R}$ and $F_{\varepsilon,\lambda}: TO_{\varepsilon, \lambda} \setminus 0 \rightarrow {\mathbb R}$ are smooth for every $\lambda \in \Lambda$ and $\varepsilon \in (0,1)$. 
Therefore $\varphi_\lambda F_{\varepsilon,\lambda}: TM\backslash 0 \rightarrow \mathbb R$ is smooth for every $\lambda$ what implies that $F_\varepsilon : TM\backslash 0 \rightarrow \mathbb R$ is smooth.
The positive homogeneity of $F_\varepsilon$ is straightforward.
Finally it is easy to see that $\Hess_y F^2_\varepsilon(x, \cdot)$ is positive definite for every $x \in M$ due to (\ref{F epsilon definida em TM}).

\

(2) Let $K_{TM}$ be a compact subset of $TM$ and consider $(x,y) \in K_{TM}$. 
Then
\begin{eqnarray}
\left| F_\varepsilon^2(x,y) - F^2(x,y)\right| & = & \left| \sum_{\lambda \in \Lambda} \varphi_\lambda(x) F_{\varepsilon,\lambda}^2 (x,y)  - \sum_{\lambda \in \Lambda} \varphi_\lambda(x) F^2(x,y) \right| \nonumber \\
& \leq &   \sum_{\lambda \in \Lambda} \varphi_\lambda(x) \left| F_{\varepsilon,\lambda}^2(x,y) - F^2(x,y) \right|. \label{limita F epsilon menos F}
\end{eqnarray}
$\pi(K_{TM})$ is covered by a finite number of open subsets $V_\lambda$ and the sum (\ref{limita F epsilon menos F}) is finite when restricted to $K_{TM}$.
The function $\vert F_{\varepsilon,\lambda}^2(x,y) - F^2(x,y) \vert$ converges uniformly to zero on $K_{TM} \cap\overline{ TV_\lambda}$ as $\varepsilon$ goes to zero due (\ref{menor que delta}), what implies that the left hand side of (\ref{limita F epsilon menos F}) has the same convergence as well.
Now proceeding as in the proof of Item (2) of Lemma \ref{propriedades da aplicacao f epsilon}, this item is settled.
\end{proof}

\section{Mollifier smoothing of Finsler structures}
\label{secao smoothing finsler}

In this section we assume that $ F: TM \rightarrow [0, \infty)$  is a Finsler structure and we study several convergence results related to $G_{\psi_\lambda(\varepsilon), \lambda}$, $F_{\varepsilon,\lambda}$, $F_\varepsilon$ and $F$ in coordinate systems.

The next lemma is a version of Lemmas \ref{lema propriedade de F_epsilon} and \ref{convergencia uniforme de F til} when $F$ is  smooth in $TM\setminus 0$.
 
\begin{lem}
\label{F epsilon diferenciavel} 
Let $(M,F)$ be a Finsler manifold and $U$ be an open subset of $M$ satisfying (\ref{definindo U}). 
Then $(\zeta_\varepsilon \ast_v F) :TU\rightarrow {\mathbb R}$ has the following properties for every $\varepsilon \in \left(0, \tau\right):$
\begin{enumerate}
\item $(\zeta_\varepsilon \ast_v F)$ is smooth on $TU\setminus 0;$
\item $(\zeta_\varepsilon \ast_v F)^{-1}(r_U)$ is a smooth hypersurface of TU, where $r_U$ is defined as (\ref{constante r_U});
\item Let $\alpha=(\alpha_1,\ldots,\alpha_n,\alpha_{n+1},\ldots,\alpha_{2n})$ be a multiindex. Then
\[
D^\alpha (\zeta_\varepsilon \ast_v F) \rightarrow D^\alpha F
\] 
uniformly on compact subsets of $TU\setminus 0$.
\end{enumerate}
\end{lem}

\begin{proof}

It is a direct consequence of Lemmas \ref{lema propriedade de F_epsilon} and \ref{convergencia uniforme de F til}.
\end{proof}

Now we study $G_\varepsilon$ in the case where $F$ is a Finsler structure. 
Let $(x_0,y_0) \in TU$ with $y_0 \in S_{G_\varepsilon}[x_0, 0, r_U]$ and $(r_0,\theta_0^2,\ldots,\theta_0^n)$ be the coordinates of $y_0$ in a coordinate system defined by (\ref{sistema induzido}). By Item (2) of Lemma \ref{F epsilon diferenciavel} and the implicit function theorem, the function
\begin{equation} 
\label{phi epsilon caso diferenciavel}
\phi_\varepsilon: U \times W_\theta \rightarrow \left(\frac{1}{2},\frac{2r_U}{r_u}\right)
\end{equation}
defined as in (\ref{aplicacao phi_epsilon})
is smooth.

Therefore given $(x,r,\theta^2,\ldots,\theta^{n }) \in U\times(0,\infty) \times W_\theta,$ we have that
\begin{equation} 
\label{G epsilon caso diferenciavel} 
G_\varepsilon (x,r,\theta ) = \frac{r r_U}{\phi_\varepsilon(x, \theta)}
\end{equation}
is a {\em Finsler structure}.

Proceeding likewise with $F$ we obtain the respective smooth function
\[
\phi : U \times W_\theta \rightarrow \left(\frac{1}{2},\frac{2r_U}{r_u}\right)
\]
such that
\begin{equation} 
\label{F caso diferenciavel} 
F (x,r,\theta) = \frac{r r_U}{\phi (x, \theta)}.
\end{equation}


\begin{lem} 
\label{convergencia phi epsilon} 
Let $K_{U\times W_\theta}$ be a compact subset of $U \times W_\theta$. 
The function $\phi_\varepsilon$ defined in (\ref{phi epsilon caso diferenciavel}) has the following properties:
\begin{enumerate}
\item $\frac{1}{\phi_\varepsilon } \rightarrow \frac{1}{\phi }$ uniformly on $K_{U\times W_\theta}$;
\item $ \phi_\varepsilon  \rightarrow \phi $ uniformly on $K_{U\times W_\theta}$; 
\item Let $\alpha=(\alpha_1,\ldots,\alpha_n,\alpha_{n+2},\ldots,\alpha_{2n})$ be a multiindex. Then
\[
D^\alpha \phi_\varepsilon = \frac{\partial^{|\alpha|}\phi_\varepsilon }{\partial(x^1)^{\alpha_1}\cdots \partial(x^n)^{\alpha_n} \partial(\theta^2)^{\alpha_{n+2}} \cdots \partial(\theta^n)^{\alpha_{2n}} }
\]
converges uniformly to
\[
 D^\alpha \phi = \frac{\partial^{|\alpha|}\phi  }{\partial(x^1)^{\alpha_1}\cdots \partial(x^n)^{\alpha_n} \partial(\theta^2)^{\alpha_{n+2}} \cdots \partial(\theta^n)^{\alpha_{2n}} }  
\] 
on $K_{U\times W_\theta}$.
\end{enumerate}
\end{lem}

\begin{proof}

\

(1) From (\ref{G epsilon caso diferenciavel}) and (\ref{F caso diferenciavel}) observe that
\[
\frac{1}{\phi_\varepsilon (x, \theta)} = G_\varepsilon (x, 1/r_U, \theta), 
\]
which converges uniformly to 
\[
\frac{1}{\phi (x,\theta)} = F(x, 1/r_U, \theta)
\]
on $K_{U\times W_\theta}$ because $G_\varepsilon$ converges uniformly to $F$ on compact subsets of $TU$.

\

(2) As $\frac{1}{\phi_\varepsilon(x,\theta)}\in (r_u/(2r_U),2)$ for every $(x,\theta) \in U\times W_\theta$, it follows that $\phi_\varepsilon  \rightarrow \phi $ uniformly on $K_{U\times W_\theta}$ due to Item (1).

\

(3) Given $(x,\theta) \in U\times W_\theta,$ we have that
\[
\frac{\partial \phi_\varepsilon(x,\theta )}{\partial \theta^i} 
=- \frac{\frac{\partial (\zeta_\varepsilon \ast_v F)(x,\phi_\varepsilon(x,\theta),\theta) }{\partial \theta^i}}{\frac{\partial (\zeta_\varepsilon \ast_v F)(x,\phi_\varepsilon(x,\theta),\theta) }{\partial r}}
\]
and
\[
\frac{\partial \phi (x,\theta)}{\partial \theta^i} =- \frac{\frac{\partial F(x,\phi(x,\theta),\theta) }{\partial \theta^i}}{\frac{\partial F(x,\phi (x,\theta),\theta) }{\partial r}}.
\]
We also have analogous expressions for
\[
\frac{\partial \phi_\varepsilon(x,\theta )}{\partial x^i} \text{ and }\frac{\partial \phi (x,\theta)}{\partial x^i}.
\] 
The proof of this item follows from Item (2), Item (3) of Lemma \ref{F epsilon diferenciavel} and calculations analogous to the proof of Item (3) of Proposition \ref{propriedades da aplicacao G epsilon}.
\end{proof}

\begin{teo} 
\label{G epsilon convergencia theta i e r} 
Let $(M,F)$ be a Finsler manifold.
Let $U$ be an open subset of $M$ as defined in (\ref{definindo U}) and $(x^1, \ldots, x^n, y^1, \ldots, y^n)$ be the coordinate system on $TU$ as defined in (\ref{sistema de coordenadas x,y}). 
Let $K_{TU\setminus 0}$ be a compact subset of $TU\backslash 0$ and $\alpha=(\alpha_1,\ldots,\alpha_n,$ $\alpha_{n+1},\ldots,\alpha_{2n})$ be a multiindex.
Then 
\begin{equation}
\label{multiindice G varepsilon}
D^\alpha G_\varepsilon = \frac{\partial^{\vert \alpha\vert} G_\varepsilon}{\partial (x^1)^{\alpha_1} \ldots \partial (x^n)^{\alpha_n} \partial (y^1)^{\alpha_{n+1}} \ldots \partial (y^n)^{\alpha_{2n}}}
\end{equation}
converges uniformly to
\begin{equation}
\label{multiindice F}
D^\alpha F = \frac{\partial^{\vert \alpha\vert} F}{\partial (x^1)^{\alpha_1} \ldots \partial (x^n)^{\alpha_n} \partial (y^1)^{\alpha_{n+1}} \ldots \partial (y^n)^{\alpha_{2n}}}
\end{equation}
on $K_{TU\setminus 0}$. 
\end{teo}

\begin{proof}
The uniform convergence of
\begin{equation}
\label{d alpha g epsilon}
D^\alpha G_\varepsilon 
= \frac{\partial^{\vert \alpha \vert} G_\varepsilon}{\partial (x^1)^{\alpha_1} \ldots \partial (x^n)^{\alpha_n} \partial r^{\alpha_{n+1}} \partial (\theta^2)^{\alpha_{n+2}} \ldots \partial (\theta^{n})^{\alpha_{2n}}}
\end{equation}
to
\begin{equation}
\label{d alpha f}
D^\alpha F
= \frac{\partial^{\vert \alpha \vert} F}{\partial (x^1)^{\alpha_1} \ldots \partial (x^n)^{\alpha_n} \partial r^{\alpha_{n+1}} \partial (\theta^2)^{\alpha_{n+2}} \ldots \partial (\theta^{n})^{\alpha_{2n}}}
\end{equation}
on compact subsets $K_{U\times C(S_\theta)}$ of $U \times C(S_\theta)$ is equivalent to the uniform convergence of (\ref{multiindice G varepsilon}) to (\ref{multiindice F}) on $K_{U\times C(S_\theta)}$ as a direct consequence of the chain rule.
This former convergence is enough to settle the theorem because a compact subset $K_{TU\setminus 0}$ of $TU\backslash 0$ can be covered by a finite number of compact subsets of type $K_{U\times C(S_\theta)}$.
Finally the uniform convergence of (\ref{d alpha g epsilon}) to (\ref{d alpha f}) follows from (\ref{G epsilon caso diferenciavel}), (\ref{F caso diferenciavel}), Proposition \ref{g epsilon converge uniforme}, Lemma \ref{F epsilon diferenciavel}, Lemma \ref{convergencia phi epsilon} and calculations analogous to the proof of Item (3) of Proposition \ref{propriedades da aplicacao G epsilon}.
\end{proof}

\begin{cor}
\label{derivada Gij} Let $(M,F)$ be a Finsler manifold.
Let $U$ be an open subset of $M$ as defined in (\ref{definindo U}) and $(x^1, \ldots, x^n, y^1, \ldots, y^n)$ be the coordinate system on $TU$ as defined in (\ref{sistema de coordenadas x,y}). 
Let $K_{TU\setminus 0}$ be a compact subset of $TU\backslash 0$ and $\alpha=(\alpha_1,\ldots,\alpha_n,$ $\alpha_{n+1},\ldots,\alpha_{2n})$ be a multiindex.
Then 
\begin{equation}
\label{multiindice G2 varepsilon}
D^\alpha G^2_\varepsilon = \frac{\partial^{\vert \alpha\vert} G^2_\varepsilon}{\partial (x^1)^{\alpha_1} \ldots \partial (x^n)^{\alpha_n} \partial (y^1)^{\alpha_{n+1}} \ldots \partial (y^n)^{\alpha_{2n}}}
\end{equation}
converges uniformly to
\begin{equation}
\label{multiindice F2}
D^\alpha F^2 = \frac{\partial^{\vert \alpha\vert} F^2}{\partial (x^1)^{\alpha_1} \ldots \partial (x^n)^{\alpha_n} \partial (y^1)^{\alpha_{n+1}} \ldots \partial (y^n)^{\alpha_{2n}}}
\end{equation}
on $K_{TU\setminus 0}$. 
In particular
\begin{equation}
\label{multiindice gijG varepsilon}
D^{\tilde \alpha} (g_{G_\varepsilon})_{ij} = \frac{\partial^{\vert \tilde\alpha\vert} (g_{G_\varepsilon})_{ij}}{\partial (x^1)^{\tilde \alpha_1} \ldots \partial (x^n)^{\tilde \alpha_n} \partial (y^1)^{\tilde \alpha_{n+1}} \ldots \partial (y^n)^{\tilde \alpha_{2n}}}
\end{equation}
converges uniformly to
\begin{equation}
\label{multiindice gij}
D^{\tilde \alpha} g_{ij} = \frac{\partial^{\vert \tilde \alpha\vert} g_{ij}}{\partial (x^1)^{\tilde \alpha_1} \ldots \partial (x^n)^{\tilde \alpha_n} \partial (y^1)^{\tilde \alpha_{n+1}} \ldots \partial (y^n)^{\tilde \alpha_{2n}}}
\end{equation}
on $K_{TU\setminus 0}$ for every multiindex $\tilde \alpha = (\tilde \alpha_1, \ldots, \tilde \alpha_{2n})$, where $(g_{G_\varepsilon})_{ij}$ are the coefficients of the fundamental tensor of $(M, G_\varepsilon)$ and $g_{ij}$ are the coefficients of the fundamental tensor of $(M,F)$. 
\end{cor}

\begin{proof}
Observe that (\ref{multiindice G2 varepsilon})  can be written as sums of products of terms in (\ref{multiindice G varepsilon}).
Analogously (\ref{multiindice F2}) can be written as the correspondent sums of products of terms in (\ref{multiindice F}).
Therefore (\ref{multiindice G2 varepsilon}) converges uniformly to (\ref{multiindice F2}) on $K_{TU\backslash 0}$. 
\end{proof}

\begin{lem} 
\label{propriedades de derivada da F check} Let $(U_\lambda,(x^i)_\lambda)$, with $\lambda \in \Lambda$, be a coordinate open subset of a Finsler manifold $(M,F)$ as defined in (\ref{define U lambda}) and $(x^1, \ldots, x^n, y^1, \ldots, y^n)_\lambda$ be the natural coordinate system on $TU_\lambda$. 
Consider a multiindex $\alpha=(\alpha_1,\ldots,\alpha_{2n})$. 
Then
\begin{equation} 
\label{convergencia derivada F2 epsilon lambda}
D^\alpha F^2_{\varepsilon,\lambda} \rightarrow D^\alpha F^2
\end{equation} 
uniformly on compact subsets $K_{TV_\lambda \setminus 0}$ of $TV_\lambda \setminus 0$.
In particular, if $(g_{\varepsilon,\lambda})_{ij}$ are the coefficients of the fundamental tensor of $F_{\varepsilon, \lambda}$ and $\tilde \alpha$ is a multiindex, then
\begin{equation} 
\label{convergencia derivada}
D^{\tilde \alpha} (g_{\varepsilon,\lambda})_{ij} \rightarrow D^{\tilde \alpha} g_{ij} 
\end{equation} 
uniformly on compact subsets $K_{TV_\lambda \setminus 0}$ of $TV_\lambda \setminus 0$, where $g_{ij}$ are the coefficients of the fundamental tensor of $F$.
\end{lem}

\begin{proof}
Consider $(x,y) \in TV_\lambda\setminus 0.$ As
\[
F^2_{\varepsilon,\lambda}(x,y)=\int \eta_{\psi_\lambda(\varepsilon)}(z)G^2_{\psi_\lambda(\varepsilon),\lambda}(x-z,y)dz,
\]
then
\[
D^\alpha (F^2_{\varepsilon,\lambda}) (x,y)
= \int \eta_{\psi_\lambda(\varepsilon)}(z) \left( D^\alpha G^2_{\psi_\lambda(\varepsilon),\lambda} \right) (x-z,y) dz.
\]
Define 
\[
(\eta_{\psi_\lambda(\varepsilon)} \ast_1 (D^\alpha F^2)) (x,y)= \int \eta_{\psi_\lambda(\varepsilon)}(z) (D^\alpha F^2)(x-z,y)dz.
\]
Now we proceed as in the proof of (\ref{menor que delta}) with $F^2_{\varepsilon, \lambda}$ and $F^2$ replaced by $D^\alpha F^2_{\varepsilon, \lambda}$ and $D^\alpha F^2$ respectively.
When we do the estimate $\vert D^\alpha F^2_{\varepsilon,\lambda}(x,y) - D^\alpha F^2(x,y)\vert < \delta$ as in (\ref{demonstracao F2 converge uniforme}), we need the uniform convergence 
\[
D^\alpha G^2_{\psi_\lambda (\varepsilon),\lambda} \rightarrow  D^\alpha F^2
\] 
on $K_{TV_\lambda \setminus 0}$, which is assured by Corollary \ref{derivada Gij}.
\end{proof}

\begin{obs}
\label{convergencia uniforme nao depende sistema de coordenadas}
Suppose that we are in the conditions of Lemma \ref{propriedades de derivada da F check}. 
Consider the Finsler structure $F_{\varepsilon,\lambda}:TV_\lambda \rightarrow [0,\infty)$ defined in terms of the coordinate system $(x^1, \ldots, x^n, y^1, \ldots,y^n)_\lambda$.
Let $(\tilde x^1, \ldots, \tilde x^n, \tilde y^1, \ldots, \tilde y^n)$ be another natural coordinate system on a neighborhood of $K_{TV_\lambda\setminus 0}$. 
Then the uniform convergence
\[
\frac{\partial^{\vert \tilde \alpha \vert}F^2_{\varepsilon, \lambda}}{\partial (\tilde x^1)^{\tilde \alpha_1} \ldots \partial (\tilde y^n)^{\tilde \alpha_{2n}}}
\rightarrow 
\frac{\partial^{\vert \tilde \alpha \vert}F^2}{\partial (\tilde x^1)^{\tilde \alpha_1} \ldots \partial (\tilde y^n)^{\tilde \alpha_{2n}}} 
\] 
on $K_{TV_\lambda \backslash 0}$ holds for every multiindex $\tilde \alpha$ because the uniform convergence
\[
\frac{\partial^{\vert \alpha \vert}F^2_{\varepsilon, \lambda}}{\partial (x^1_\lambda)^{\alpha_1} \ldots \partial (y^n_\lambda)^{\alpha_{2n}}}
\rightarrow 
\frac{\partial^{\vert \alpha \vert}F^2}{\partial (x^1_\lambda)^{\alpha_1} \ldots \partial (y^n_\lambda)^{\alpha_{2n}}} 
\]
on $K_{TV_\lambda \backslash 0}$ holds for every multiindex $\alpha$ due to the chain rule and the compactness of $K_{TV_\lambda \backslash 0}$.
In particular, the correspondent convergences of the coefficients of the fundamental tensor and their derivatives hold. 
\end{obs}


\begin{teo} 
\label{propriedades de derivada da F epsilon} 
Let $(M,F)$ be a Finsler manifold and $F_\varepsilon:TM \rightarrow [0,\infty)$ be a mollifier smoothing of $F$. 
Let $U$ be an open subset of $M$ with coordinate system $(x^1,\ldots, x^n)$ and  $(x^1, \ldots, x^n, y^1, \ldots, y^n)$ be the correspondent natural coordinate system of $TU$. 
Let $g_{ij}$ and $(g_\varepsilon)_{ij}$ be the coefficients of the fundamental tensor of $F$ and $F_\varepsilon$ respectively with respect to $(x^1, \ldots, x^n, y^1, \ldots, y^n)$.
Then 
\begin{equation}
D^\alpha (g_\varepsilon)_{ij} \rightarrow D^\alpha g_{ij} \label{convergencia uniforme tensor fundamental}
\end{equation}
and
\begin{equation}
D^\alpha (g_\varepsilon)^{ij} \rightarrow D^\alpha g^{ij} \label{convergencia uniforme tensor inverso}
\end{equation}
uniformly on compact subsets of $TU\setminus 0$, where $\alpha=(\alpha_1,\ldots,\alpha_{2n}) $ is a multiindex and the partial derivatives are taken with respect to $(x^1, \ldots, x^n, y^1, \ldots,$ $y^n)$.
\end{teo}

\begin{proof}
Convergence (\ref{convergencia uniforme tensor inverso}) follows from (\ref{convergencia uniforme tensor fundamental}) and the formula of the inverse of a matrix in terms of its adjoint and its determinant.

Let us prove (\ref{convergencia uniforme tensor fundamental}).
Let $K_{TU\setminus 0}$ be a compact subset of $TU\backslash 0$. 
If $(x,y) \in K_{TU\setminus 0}$, then
\[
D^\alpha (g_\varepsilon)_{ij}(x,y) = D^\alpha\left(\sum_{\lambda \in \Theta}\varphi_\lambda(x) (g_{\varepsilon,\lambda})_{ij}(x,y) \right)
\]
as a consequence of (\ref{F epsilon definida em TM}).
It follows that
\begin{eqnarray}
D^\alpha (g_\varepsilon)_{ij}(x,y) & = & \sum_{\lambda \in \Theta} D^\alpha\left(\varphi_\lambda(x)  (g_{\varepsilon,\lambda})_{ij}(x,y) \right) \label{decompoe d alpha f epsilon} \\
& = &  \sum_{\lambda \in \Theta} \sum_{\upsilon \leq \alpha} D^\upsilon \varphi_\lambda(x)D^{\alpha-\upsilon}(g_{\varepsilon,\lambda})_{ij}(x,y), \nonumber
\end{eqnarray}
where $\lambda \in \Theta$ are the indexes such that $K_{TU \setminus 0} \cap \overline{TV_\lambda} \neq \emptyset$. 
Then $\Theta$ is a finite set and it follows that  
\[
\sum_{\lambda \in \Theta} \sum_{\upsilon\leq \alpha} D^\upsilon \varphi_\lambda(x)D^{\alpha-\upsilon} (g_{\varepsilon,\lambda})_{ij}(x,y)
\] 
converges uniformly to 
\[
\sum_{\lambda \in \Theta} \sum_{\upsilon\leq \alpha} D^\upsilon \varphi_\lambda(x)D^{\alpha-\upsilon}g_{ij}(x,y) \label{derivada F decomposto}
\]
on $K_{TU \setminus 0}$ due to Lemma \ref{propriedades de derivada da F check} and Remark \ref{convergencia uniforme nao depende sistema de coordenadas}. 
But (\ref{derivada F decomposto}) is equal to
\[
\sum_{\upsilon\leq \alpha} D^{\alpha-\upsilon}g_{ij}(x,y) D^\upsilon \left( \sum_{\lambda \in \Theta} \varphi_\lambda(x)\right)
= \sum_{\upsilon\leq \alpha} D^{\alpha-\upsilon}g_{ij}(x,y) D^\upsilon \left( 1 \right) = D^\alpha g_{ij}(x,y),
\]
what settles the theorem.
\end{proof}

\section{Convergence of connections and curvatures}
\label{secao convergencia conexoes curvatura}

Let $(M,F)$ be a Finsler manifold. 
In this section we study the convergence of the Chern connection, Cartan connection, Hashiguchi connection, Berwald connection and the flag curvature of the mollifier smoothing $F_\varepsilon$ to the corresponding objects of $F$. 
The precise nature of these convergences will be explained afterwards.
We represent the geometrical objects with respect to $F_\varepsilon$ with a subscript $\varepsilon$, as we did in Theorem \ref{propriedades de derivada da F epsilon} for the fundamental tensor $g_\varepsilon$. 

Let $(x^1, \ldots,$ $x^n)$ be a coordinate system on an open subset $U$ of $M$ and $(x^1, \ldots, x^n, y^1, \ldots,$ $y^n)$ be the corresponding natural coordinate system on $TU$. 
In this section, {\em whenever we are studying components of a geometrical object, we suppose that these coordinate open subsets are in place}.

\begin{lem}
\label{several uniform convergences}
The following geometrical objects of $(M,F_\varepsilon)$ converges uniformly to the respective objects of $(M,F)$ on compact subsets of $TU\backslash 0$:
\begin{eqnarray}
D^\alpha (\ell_\varepsilon)^i & \rightarrow & D^\alpha \ell^i; \nonumber \\
D^\alpha (A_\varepsilon)_{ijk} & \rightarrow & D^\alpha A_{ijk}; \nonumber \\
D^\alpha (C_\varepsilon)_{ijk} & \rightarrow & D^\alpha C_{ijk}; \nonumber \\
D^\alpha (\gamma_\varepsilon)^i{}_{jk} & \rightarrow & D^\alpha \gamma^i{}_{jk}; \nonumber \\
D^\alpha (N_\varepsilon)^i{}_j & \rightarrow & D^\alpha N^i{}_j; \nonumber \\
D^\alpha (\Gamma_\varepsilon)^i{}_{jk} & \rightarrow & D^\alpha \Gamma^i{}_{jk}; \nonumber \\
D^\alpha (\sigma_\varepsilon)^i & \rightarrow & D^\alpha \sigma^i; \nonumber \\
D^\alpha (\dot A_\varepsilon)^i{}_{jk} & \rightarrow & D^\alpha \dot A^i{}_{jk}; \nonumber \\
D^\alpha (R_\varepsilon)_{jikl} & \rightarrow & D^\alpha R_{jikl}, \nonumber
\end{eqnarray}
for every multiindex $\alpha$.
\end{lem}

\begin{proof}
The proof is a direct consequence of (\ref{define distinguished}), (\ref{define componente tensor de cartan}), (\ref{unormalized Cartan tensor}), (\ref{definicao formal Christoffel simbols}), (\ref{definicao non linear connection}), (\ref{define delta delta}), (\ref{Gamma almost metric compatibility}), (\ref{define a dot}), (\ref{definicao Gi}),  (\ref{Rjikl hh}), (\ref{Rjikl abaixa}), Theorem \ref{propriedades da aplicacao F epsilon definida em TM}  and Theorem \ref{propriedades de derivada da F epsilon}.
\end{proof}

Now we present the concept of uniform convergence on $\pi^\ast TM$.

\begin{defi}
\label{define estrutura de Finsler C0 piTM}
A $C^0$-Finsler structure on $\pi^\ast TM$ is a continuous function $F_\pi : \pi^\ast TM \rightarrow \mathbb R$ such that $F_\pi(x,y,\cdot):(\pi^\ast TM)_{(x,y)} \rightarrow \mathbb R$ is an asymmetric norm on $(\pi^\ast TM)_{(x,y)} \cong T_xM$.
\end{defi}

\begin{defi} 
\label{define convergencia uniforme piTM}
Let $M$ be a differentiable manifold and $K_{TM \setminus 0}$ be a compact subset of $TM \setminus 0$.
We say that a one parameter family of continuous sections $E_\varepsilon$ converges uniformly to a continuous section $E$ on $K_{TM \setminus 0}$ if there exist a $C^0$-Finsler structure $F_\pi$ on $\pi^\ast TM$ such that for every $\delta >0$, there exist a $\beta >0$ such that 
\[
F_\pi (x, y, E_\varepsilon(x,y) - E(x,y))< \delta
\]
for every $(x,y) \in K_{TM \setminus 0}$ and every $\varepsilon \in (0, \beta)$.
\end{defi}

\begin{obs}
\label{observacao nao depende estrutura finsler}
The definition of uniform convergence stated in Definition \ref{define convergencia uniforme piTM} does not depend on the choice of $F_\pi$.
In fact, if $\tilde F_\pi$ is another $C^0$-Finsler structure on $\pi^\ast TM$ and $K_{TM \setminus 0}$ is a compact subset in $TM\setminus 0$, then
there exist $c, C >0$ such that
\[
c F_\pi (x,y, E(x,y)) \leq \tilde F_\pi (x,y, E(x,y)) \leq C F_\pi (x,y, E(x,y))
\]
for every $(x,y) \in K_{TM \setminus 0}$ and every continuous section $E$ on $\pi^\ast TM$.
\end{obs}

In what follows, we denote by $\nabla_\varepsilon$ the Chern connection on $\pi^\ast TM$ with respect to $F_\varepsilon$.

\begin{teo} 
\label{convergencia uniforme Chern}
Let $(M,F)$ be a Finsler manifold and let $F_\varepsilon$ be a mollifier smoothing of $F$. 
Let $X$ be an arbitrary section on $TM\backslash 0$ and $E$ be an arbitrary section on $\pi^*TM$.
Then $\left( \nabla_\varepsilon \right)_X E$
converges uniformly to $\nabla_X E$ on compact subsets $K_{TM \setminus 0}$ of $TM\setminus 0$.
The statement also holds for the Cartan connection, Hashiguchi connection and Berwald connection.
\end{teo}

\begin{proof}
Notice that it is enough to prove the uniform convergence for compact subsets $K_{TU \setminus 0}$ of any coordinate open subset $TU\setminus 0$ because a finite number of them cover $K_{TM \setminus 0}$.
Moreover it is enough to prove the uniform convergence with respect to the fundamental tensor $g$ on $\pi^\ast TM$ due to Remark \ref{observacao nao depende estrutura finsler}.

Consider a natural coordinate system $(x^1, \ldots, x^n, y^1, \ldots, y^n)$ on $TU$. 
Then
\begin{eqnarray*}
\nabla_X E & = & \nabla_{\left( X^i \frac{\partial}{\partial x^i} + Y^j \frac{\partial}{\partial y^j}\right)}\left(E^k\frac{\partial}{\partial x^k}\right) \\
&= & X^i \frac{\partial E^k}{\partial x^i} \frac{\partial }{\partial x^k} + Y^j \frac{\partial E^k}{\partial y^j} \frac{\partial}{\partial x^k} + E^k X^i \omega_k{}^j \left( \frac{\partial}{\partial x^i}\right) \frac{\partial}{\partial x^j} \\
& = & X^i \frac{\partial E^k}{\partial x^i} \frac{\partial }{\partial x^k} + Y^j \frac{\partial E^k}{\partial y^j} \frac{\partial}{\partial x^k} + E^j X^i \Gamma^k{}_{ji} \frac{\partial}{\partial x^k},
\end{eqnarray*}
and
\begin{eqnarray*}
& \sqrt{ g_{(x,y)}\left(\nabla_X E - (\nabla_\varepsilon)_X E, \nabla_X E - (\nabla_\varepsilon)_X E\right) } &  \\
& = \sqrt{ g_{(x,y)}\left( E^j X^i (\Gamma^k{}_{ji} - (\Gamma_\varepsilon)^k{}_{ji})\frac{\partial}{\partial x^k}, E^{j^\prime} X^{i^\prime} (\Gamma^{k^\prime}{}_{j^\prime i^\prime} - (\Gamma_\varepsilon)^{k^\prime}{}_{j^{\prime}i^\prime})\frac{\partial}{\partial x^{k^\prime}}\right)} & \\
& = \sqrt{ g_{kk^\prime (x,y)} E^j X^i ((\Gamma^k{}_{ji} - (\Gamma_\varepsilon)^k{}_{ji}) E^{j^\prime} X^{i^\prime} (\Gamma^{k^\prime}{}_{j^\prime i^\prime} - (\Gamma_\varepsilon)^{k^\prime}{}_{j^{\prime}i^\prime})} \rightarrow 0&
\end{eqnarray*}
uniformly on $K_{TU \setminus 0}$ as $\varepsilon \rightarrow 0$.
Therefore $(\nabla_\varepsilon)_X E$ converges uniformly to $\nabla_X E$ on $K_{TU \setminus 0}$ due to Lemma \ref{several uniform convergences}.
This proves the theorem for the Chern connection.

The proof for the other connections are similar due to Definition \ref{Cartan Hasiguchi Berwarld} and Lemma \ref{several uniform convergences}.
\end{proof}

Now we will study the uniform convergence of flag curvatures of $(M,F_\varepsilon)$ to $(M,F)$.
Denote the set of flags on $M$ by $\mathcal F$ and the set of flags on the tangent space $T_xM$ by $\mathcal F_x$.
The flag curvature of $(M,F)$ and $(M,F_\varepsilon)$ will be denoted by $K$ and $K_\varepsilon$ respectively.

\begin{teo}
\label{curvatura flag converge uniforme}
Let $(M,F)$ be a Finsler manifold and let $F_\varepsilon$ be a mollifier smoothing of $F$.
Let $K_M$ be a compact subset of $M$ and $\delta >0$.
Then there exist $\beta > 0$ such that
\begin{equation}
\label{equacao convergencia uniforme curvatura flag}
\sup_{x \in K_M} \sup_{(y,z) \in \mathcal F_x} \vert K(x,y,z) - K_\varepsilon(x,y,z) \vert < \delta
\end{equation}
for every $\varepsilon \in (0, \beta)$.
\end{teo}

\begin{proof}
Notice that it is enough to prove (\ref{equacao convergencia uniforme curvatura flag}) for compact subsets $K_U$ of a coordinate open subset $U \subset M$, because a finite number of them can cover a general compact subset $K_M \subset M$. 

Consider $\pi^\ast TU$ with the coordinate system $(x^1, \ldots, x^n,$ $y^1, \ldots, y^n, z^1,$ $\ldots, z^n)$,
where $(x^1, \ldots,$ $x^n,$ $y^1, \ldots, y^n)$ are the natural coordinates on $TU\backslash 0$ and the elements of the fibers of $\pi^\ast TU$ are given by $z^i \frac{\partial }{\partial x^i}$.
We endow each fiber of $\pi^\ast TU$ with the Euclidean inner product with respect to the basis $(\partial / \partial x^1, \ldots, \partial / \partial x^n)$.
Every flag $(x,(y,z))$, with $x \in K_U$ and $(y,z) \in \mathcal F_x$, can be identified with two points in 
\[
K_{\pi^\ast TU} =\left\{ (x, y, z); \sum_{i=1}^n (y^i)^2 = \sum_{i=1}^n (z^i)^2 = 1, \sum_{i=1}^n y^i z^i = 0\right\} \subset \pi^\ast TU.
\]
Reciprocally, every pair of points $(x,y,$ $\pm z)$ is identified with the flag $(y,\pm z)$ on $T_xU$. 
It is straightforward that $K_{\pi^\ast TU}$ is compact.
Therefore
\begin{equation}
\sup_{x \in K_M} \sup_{(y,z) \in \mathcal F_x} \vert K(x,y,z) - K_\varepsilon(x,y,z) \vert = \sup_{(x,y,z) \in K_{\pi^\ast TU}}\vert K(x,y,z) - K_\varepsilon(x,y,z) \vert \label{sup converge uniforme}
\end{equation}
where
\begin{equation}
K(x,y,z) = \frac{z^i(y^jR_{jikl}y^l)z^k}{g_{\tilde \imath \tilde \jmath}  g_{\tilde k \tilde l}y^{\tilde \imath} y^{\tilde \jmath} z^{\tilde k} z^{\tilde l}-[g_{\hat \imath \hat \jmath}y^{\hat \imath} z^{\hat \jmath}]^2} \nonumber
\end{equation}
and
\begin{equation}
K_\varepsilon (x,y,z) = \frac{z^i(y^j(R_\varepsilon)_{jikl}y^l)z^k}{(g_\varepsilon)_{\tilde \imath \tilde \jmath}  (g_\varepsilon)_{\tilde k \tilde l}y^{\tilde \imath} y^{\tilde \jmath} z^{\tilde k} z^{\tilde l}-[(g_\varepsilon)_{\hat \imath \hat \jmath}y^{\hat \imath} z^{\hat \jmath}]^2}. \nonumber
\end{equation}
Therefore (\ref{sup converge uniforme}) converges uniformly to zero on $K_{\pi^\ast TU}$ as $\varepsilon$ goes to zero because the uniform convergences
\[
(g_\varepsilon)_{ij} \rightarrow g_{ij}
\]
and
\[
(R_\varepsilon)_{jikl} \rightarrow R_{jikl}
\]
hold on $\pi (K_{\pi^\ast TU})$ due to Theorem \ref{propriedades de derivada da F epsilon} and  Lemma \ref{several uniform convergences}.
\end{proof}
 
\section{Examples}
\label{secao exemplos}

The aim of this section is to present examples of piecewise Finsler manifolds and show how calculations can be made in this case. 
Our main example is piecewise Riemannian for the sake of simplicity, but similar calculations can be made for the piecewise Finsler case. 
We comment about the necessary adaptations from the Riemannian case to the Finsler case in Remark \ref{o exemplo Finsler}. 
In this example we use only the horizontal smoothing $F_\varepsilon$ of $F$ because the latter is already vertically smooth.

Let $(M,F)$ be a $C^0$-Finsler manifold such that $F(x,\cdot):T_xM \rightarrow \mathbb R$ is a Minikowski norm for every $x\in M$.
If we consider the mollifier smoothing $F_\varepsilon$ of $F$ given by (\ref{F check lambda}) and (\ref{F epsilon definida em TM}), with $G_{\psi_\lambda(\varepsilon),\lambda}$ replaced by $F\vert_{TU_\lambda}$, then Lemma \ref{propriedades da aplicacao f epsilon} and Theorem \ref{propriedades da aplicacao F epsilon definida em TM} hold.
Moreover if we suppose that $F$ is a Finsler structure, then Theorem \ref{propriedades de derivada da F epsilon}, Lemma \ref{several uniform convergences}, Theorem \ref{convergencia uniforme Chern} and Theorem  \ref{curvatura flag converge uniforme} hold as well.
The proof of these results are very similar to the original ones and will be omitted here.

Let $M = \mathbb R^n$, $\bar M_+ = \overline{\mathbb R^n_+} = \{(x^1, \ldots, x^n) \in \mathbb R^n;x^n \geq 0\}$ and $\bar M_- = \overline{\mathbb R^n_-} = \{(x^1, \ldots, x^n) \in \mathbb R^n;x^n \leq 0\}$.
Let $g$ be a piecewise smooth Riemannian metric on $M$ such that $g_+ := g\vert_{\bar M_+}$ and $g_- = g\vert_{\bar M_-}$ admit smooth extensions in a neighborhood of $\bar M_+$ and $\bar M_-$ respectively.
$TM$ will be endowed with its canonical coordinate system $(x^1, \ldots, x^n, y^1, \ldots, y^n)$ and the Riemannian metric can be considered as the Finsler structure
\begin{equation}
\label{piecewise smooth Riemannian}
F(x^1, \ldots, x^n, y^1, \ldots, y^n) = \sqrt{g_{ij}(x^1, \ldots, x^n) y^i y^j}.
\end{equation}
The Riemannian metric is the fundamental tensor of $F$ and it doesn't depend on $y$.

The locally finite differentiable structure $\{(U_\lambda,(x^i)_\lambda)\}_{\lambda\in\Lambda}$ is given by the unique element $(\mathbb R^2, \mathrm{id})$.
The partition of the unity is given by $\{ \varphi \equiv 1\}$ and we define
\[
F^2_\varepsilon (x,y) = \int_M \eta_\varepsilon (x-z)F^2(z,y) dz. 
\]
Observe that the fundamental tensor $g_\varepsilon$ of $F_\varepsilon$ has its components given by
\begin{equation}
(g_\varepsilon)_{ij}(x)= \int_M \eta_\varepsilon (x-z)g_{ij}(z) dz. \label{termo nao pico gij}
\end{equation}
{\em Through this section, with the exception of Remark \ref{o exemplo Finsler}, $(M,g)$ stands for this example and $(M,g_\varepsilon)$ is given by (\ref{termo nao pico gij}).}
All the theory presented in this work hold for this example because we can restrict the calculations on open subsets of $M$ or $TM\setminus 0$ with compact closure.
As in the proof of Lemma \ref{lema propriedade de F_epsilon}, denote $x=(x^\prime, x^n)$, where $x^\prime = (x^1, \ldots,$  $x^{n-1})$.
The next proposition gives the main formulas in order to study the geometry of $(M,g_\varepsilon)$ in a neighborhood of $x^n=0$. 

\begin{pro}
\label{lema fundamental exemplo} 
For the piecewise smooth Riemannian manifold $(M, g)$ defined in (\ref{piecewise smooth Riemannian}), the following formulas hold for every $i,j \in \{1, \ldots, n\}$:
\begin{eqnarray}
\frac{\partial (g_\varepsilon)_{ij}}{\partial x^k}(x)
& = & 
\left( \eta_\varepsilon \ast \frac{\partial g_{ij}}{\partial x^k} \right) (x) \nonumber \\
& = & \int_{\bar M_-} \eta_\varepsilon (x - z) \frac{\partial g_{ij-}}{\partial x^k}(z)dz
+ \int_{\bar M_+} \eta_\varepsilon (x - z) \frac{\partial g_{ij+}}{\partial x^k}(z)dz \label{termo nao pico k}
\end{eqnarray}
for every $k\in \{1, \ldots, n \}$,
\begin{eqnarray}
\frac{\partial^2 (g_\varepsilon)_{ij}}{\partial x^l \partial x^k}(x)
& = & \left(\eta_\varepsilon \ast \frac{\partial^2 g_{ij}}{\partial x^l \partial x^k}\right) (x) \nonumber \\
& = & \int_{\bar M_-} \eta_\varepsilon (x - z) \frac{\partial^2 g_{ij-}}{\partial x^l \partial x^k}(z)dz
+ \int_{\bar M_+} \eta_\varepsilon (x - z) \frac{\partial g_{ij+}}{\partial x^l \partial x^k}(z)dz \label{termo nao pico lk}
\end{eqnarray}
if $(i,j) \neq (n,n)$ and 
\begin{eqnarray}
\frac{\partial^2 (g_\varepsilon)_{ij}}{\partial (x^n)^2} (x^\prime, x^n) & = & \left(\eta_\varepsilon \ast \frac{\partial^2 g_{ij}}{\partial (x^n)^2}\right) (x^\prime, x^n) \nonumber \\ 
& & + \int_{\{x^n = 0\}} \eta_\varepsilon (x^\prime - z^\prime,x^n) \left( \frac{\partial g_{ij+}}{\partial x^n}(z^\prime, 0) - \frac{\partial g_{ij-}}{\partial x^n}(z^\prime, 0) \right) dz^\prime \nonumber \\
& = & \int_{\bar M_-} \eta_\varepsilon (x - z) \frac{\partial^2 g_{ij-}}{\partial (x^n)^2}(z)dz + \int_{\bar M_+} \eta_\varepsilon (x - z) \frac{\partial^2 g_{ij+}}{\partial (x^n)^2}(z)dz  \nonumber \\
& & + \int_{\{x^n = 0\}} \eta_\varepsilon (x^\prime - z^\prime,x^n) \left( \frac{\partial g_{ij+}}{\partial x^n}(z^\prime, 0) - \frac{\partial g_{ij-}}{\partial x^n}(z^\prime, 0) \right) dz^\prime. \label{termo pico}
\end{eqnarray}
\end{pro}

\begin{proof}
For $k \in \{1, \ldots, n-1\}$, we have that
\[
\frac{\partial (g_\varepsilon)_{ij}}{\partial x^k}(x) = \frac{\partial}{\partial x^k}\int_{\mathbb R^n} \eta_\varepsilon (z) g_{ij}(x - z)dz
= \int_{\mathbb R^n} \eta_\varepsilon (z) \frac{\partial g_{ij}}{\partial x^k}(x - z) dz = \left( \eta_\varepsilon \ast \frac{\partial g_{ij}}{\partial x^k} \right) (x)
\]
because $\frac{\partial g_{ij}}{\partial x^k} $ is continuous.
For $k=n$, we have that
\begin{eqnarray}
\frac{\partial (g_\varepsilon)_{ij}}{\partial x^n}(x^\prime, x^n) & = & \frac{\partial}{\partial x^n}\int_{\mathbb R^n} \eta_\varepsilon (x^\prime - z^\prime, x^n - z^n) g_{ij}(z^\prime, z^n) dz \nonumber \\
& = & \int_{\mathbb R^n} \frac{\partial \eta_\varepsilon}{ \partial x^n} (x^\prime - z^\prime, x^n - z^n) g_{ij}(z^\prime, z^n)dz \nonumber \\
& = & -\int_{\mathbb R^{n-1}} \int_{\mathbb R} \frac{\partial \eta_\varepsilon}{ \partial z^n} (x^\prime - z^\prime, x^n - z^n) g_{ij}(z^\prime, z^n)dz^n dz^\prime \nonumber \\
& = & -\int_{\mathbb R^{n-1}} \int_{(-\infty, 0]} \frac{\partial \eta_\varepsilon}{ \partial z^n} (x^\prime - z^\prime, x^n - z^n) g_{ij-}(z^\prime, z^n)dz^n dz^\prime \nonumber \\
& & - \int_{\mathbb R^{n-1}} \int_{[0,\infty)} \frac{\partial \eta_\varepsilon}{ \partial z^n} (x^\prime - z^\prime, x^n - z^n) g_{ij+}(z^\prime, z^n)dz^n dz^\prime. \nonumber 
\end{eqnarray}
Integrating by parts, we get
\begin{eqnarray}
\frac{\partial (g_\varepsilon)_{ij}}{\partial x^n}(x^\prime, x^n) 
& = & \int_{\mathbb R^{n-1}} \int_{(-\infty, 0]} \eta_\varepsilon (x^\prime - z^\prime, x^n - z^n) \frac{\partial g_{ij-}}{\partial z^n}(z^\prime, z^n)dz^n dz^\prime \nonumber \\
& & - \int_{\mathbb R^{n-1}} \eta_\varepsilon (x^\prime - z^\prime, x^n) g_{ij-} (z^\prime, 0) dz^\prime \nonumber \\
& & + \int_{\mathbb R^{n-1}} \int_{[0,\infty)} \eta_\varepsilon (x^\prime - z^\prime, x^n - z^n) \frac{\partial g_{ij+}}{\partial z^n}(z^\prime, z^n) dz^n dz^\prime \nonumber \\
& & + \int_{\mathbb R^{n-1}} \eta_\varepsilon (x^\prime - z^\prime, x^n) g_{ij+} (z^\prime, 0) dz^\prime \nonumber \\
& = & \int_{\mathbb R^n} \eta_\varepsilon (x^\prime - z^\prime, x^n - z^n) \frac{\partial g_{ij}}{\partial z^n}(z^\prime, z^n)dz^n dz^\prime \nonumber \\
& & + \int_{\mathbb R^{n-1}} \eta_\varepsilon (x^\prime - z^\prime, x^n) (g_{ij+} (z^\prime, 0) - g_{ij-} (z^\prime, 0)) dz^\prime \nonumber \\
& = & \left(\eta_\varepsilon \ast \frac{\partial g_{ij}}{\partial x^n}\right) (x), \nonumber
\end{eqnarray}
because $g_{ij+} = g_{ij-}$ on $x^n=0$.
Here it is worth to emphasize that $\frac{\partial g_{ij}}{\partial x^n}$ isn't necessarily continuous on $x^n = 0$.

The other cases follows likewise, splitting the domain where the integrand is continuous, taking the derivative of $\eta_\varepsilon$ inside the integral, changing the derivative of $\eta_\varepsilon$ from variable ``$x$'' to ``$z$'', using Fubini's theorem and using integration by parts.
\end{proof}

We are interested to study the behavior of the sectional curvatures of $(M,g_\varepsilon)$ on $x^n=0$ when $\varepsilon$ goes to zero. 
In order to simplify the analysis, we suppose that $n=2$.
Remarks for more general cases will be made in Remark \ref{o exemplo Finsler}. 
The sectional curvature of $(M,g_\varepsilon)$ is given by
\begin{equation}
K_\varepsilon 
= \frac{(R_\varepsilon)_{1221}}{(g_\varepsilon)_{11} (g_\varepsilon)_{22} - (g_\varepsilon)_{12}^2} 
= \frac{(g_\varepsilon)_{2i}(R_\varepsilon)_1{}^i{}_{21}}{(g_\varepsilon)_{11} (g_\varepsilon)_{22} - (g_\varepsilon)_{12}^2}\label{curvatura secional varepsilon exemplo}
\end{equation}
where
\[
(R_\varepsilon)_1{}^i{}_{21} 
= \frac{\partial (\Gamma_\varepsilon)^i{}_{11}}{\partial x^2} - \frac{\partial (\Gamma_\varepsilon)^i{}_{12}}{\partial x^1} + (\Gamma_\varepsilon)^i{}_{j2}(\Gamma_\varepsilon)^j{}_{11} - (\Gamma_\varepsilon)^i{}_{j1}(\Gamma_\varepsilon)^j{}_{12}
\]
and
\[
(\Gamma_\varepsilon)^i{}_{jl} = \frac{(g_\varepsilon)^{im}}{2}\left(\frac{\partial (g_\varepsilon)_{mj}}{\partial x^l} - \frac{\partial (g_\varepsilon)_{jl}}{\partial x^m} + \frac{\partial (g_\varepsilon)_{ml}}{\partial x^j}\right).
\]
Therefore $K_\varepsilon$ is given in terms of:
\begin{enumerate}
\item $(g_\varepsilon)_{ij}$ and their derivatives of order up to two;
\item $(g_\varepsilon)^{ij}$ and their derivatives of order up to one.
\end{enumerate}

Let $U$ be a subset of $M$ with compact closure. Then 
\[
\max\left\{ \sup_{\substack{\varepsilon \in (0,1) \\ x \in U}} \left\vert (g_\varepsilon)_{ij}(x) \right\vert, \sup_{\substack{\varepsilon \in (0,1) \\ x \in U \\ k\in \{1, 2\}}} \left\vert  \frac{\partial (g_\varepsilon)_{ij}}{\partial x^k}(x) \right\vert, \sup_{\substack{\varepsilon \in (0,1) \\ x \in U \\ (l,k)\neq (2, 2)}} \left\vert \frac{\partial^2 (g_\varepsilon)_{ij}}{\partial x^l \partial x^k} (x)\right\vert\right\} < \infty
\]
for every $i,j\in \{1, 2\}$ due to (\ref{termo nao pico gij}), (\ref{termo nao pico k}), (\ref{termo nao pico lk}) and the smoothness of these terms in $\bar M_+$ and $\bar M_-$. 
In addition, the uniform convergence $\lim\limits_{\varepsilon \rightarrow 0} (g_\varepsilon)_{ij}(x)$ $ = g_{ij}(x)$ on $U$ implies that $(g_\varepsilon)^{ij}$ converges uniformly to $g^{ij}$ on $U$ when $\varepsilon$ converges to zero. 
Therefore there exist $\varepsilon^\prime >0$ such that
\[
\sup_{\substack{\varepsilon \in (0,\varepsilon^\prime) \\ x \in U}} \left\vert (g_\varepsilon)^{ij}(x) \right\vert < \infty
\]
for every $i,j\in \{1,2\}$.

Finally we have that
\[
\frac{\partial }{\partial x^k}(g_\varepsilon)^{ij} 
= - (g_\varepsilon)^{il} \left( \frac{\partial}{\partial x^k} (g_\varepsilon)_{lm}\right) (g_\varepsilon)^{mj}
\]
as a consequence of
\[
0 = \frac{\partial}{\partial x^k}\left( (g_\varepsilon)^{il}(g_\varepsilon)_{lm} \right),
\]
and it follows that
\[
\sup_{\substack{\varepsilon \in (0,\varepsilon^\prime) \\ x \in U}} \left\vert \frac{\partial}{\partial x^k}(g_\varepsilon)^{ij}(x) \right\vert < \infty
\]
for every $i, j, k \in \{1,2\}$.

The only term of $(R_\varepsilon)^{\ i}_{1\ 21}$ which eventually isn't bounded on $U$ when $\varepsilon$ goes to zero is (\ref{termo pico}) of $\frac{\partial^2 (g_\varepsilon)_{ij}}{\partial (x^n)^2} (x)$.
This term can go to $\pm \infty$ when $\varepsilon$ goes to zero.
In fact
\begin{eqnarray}
& & \lim_{\varepsilon \rightarrow 0} \int_{\{x^2 = 0\}} \eta_\varepsilon (x^\prime - z^\prime, 0) \left( \frac{\partial g_{ij+}}{\partial z^n}(z^\prime, 0) - \frac{\partial g_{ij-}}{\partial z^n}(z^\prime, 0) \right) dz^\prime \nonumber \\
& = & \lim_{\varepsilon \rightarrow 0} \int_{\{x^2 = 0\}} \frac{\eta \left(\frac{x^\prime - z^\prime, 0}{\varepsilon}\right)}{\varepsilon^n} \left( \frac{\partial g_{ij+}}{\partial z^n}(z^\prime, 0) - \frac{\partial g_{ij-}}{\partial z^n}(z^\prime, 0) \right) dz^\prime. \label{desenvolve termo pico}
\end{eqnarray}
If we denote
\[
c = \int_{\{x^2 = 0\}} \frac{\eta \left( \frac{x^\prime - z^\prime, 0}{\varepsilon} \right)}{\varepsilon^{n-1}}dz^\prime,
\]
this integral doesn't depend on $\varepsilon$ and
\begin{eqnarray}
& & \lim_{\varepsilon \rightarrow 0} \int_{\{x^2 = 0\}} \frac{\eta \left(\frac{x^\prime - z^\prime, 0}{\varepsilon}\right)}{\varepsilon^{n-1}} \left( \frac{\partial g_{ij+}}{\partial z^n}(z^\prime, 0) - \frac{\partial g_{ij-}}{\partial z^n}(z^\prime, 0) \right) dz^\prime \nonumber \\
& = & c \left( \frac{\partial g_{ij+}}{\partial x^n}(x^\prime, 0) - \frac{\partial g_{ij-}}{\partial x^n}(x^\prime, 0) \right). \nonumber
\end{eqnarray}
Therefore (\ref{desenvolve termo pico}) is $\infty$ or $-\infty$ depending if
\[
\left( \frac{\partial g_{ij+}}{\partial x^n}(x^\prime, 0) - \frac{\partial g_{ij-}}{\partial x^n}(x^\prime, 0) \right)
\]
is strictly positive or strictly negative respectively.

Let us study the influence of (\ref{termo pico}) on the total curvature of $M$. 
Direct calculations show that this term appear as
\[
- \frac{1}{2\det g_\varepsilon}\frac{\partial^2 (g_\varepsilon)_{11}}{\partial (x^2)^2} 
\]
in (\ref{curvatura secional varepsilon exemplo}).

Denote
\begin{equation}
\label{gap g112}
q(z^1) := \frac{\partial (g_\varepsilon)_{11+}}{\partial z^2}(z^1, 0) - \frac{\partial (g_\varepsilon)_{11-}}{\partial z^2}(z^1, 0). \end{equation}
Fix a line segment $I = [a,b] \times \{0\}$ in $\{x^2 = 0\}$.
The influence of the gap (\ref{gap g112}) for $(z^1,0)$ varying along $I$ on the total curvature of $(M,g)$ is given by 
\begin{eqnarray}
& & -\int_{I} \left( \int_M \frac{\eta_\varepsilon (x^1 - z^1, x^2)}{2\det g_\varepsilon} q(z^1) \sqrt{\det g_\varepsilon} .dx^1 . dx^2 \right) dz^1 \nonumber \\
& = & -\int_{I} \left( \int_M \frac{\eta_\varepsilon (x^1 - z^1, x^2)}{2\sqrt{(g_\varepsilon)_{11}.\det g_\varepsilon}} q(z^1) .dx^1 . dx^2 \right) ds, \label{curvatura integra do pico}
\end{eqnarray}
where $\sqrt{\det g_\varepsilon} dx^1 dx^2$ is the volume element of $(M,g_\varepsilon)$ and $ds = \sqrt{(g_\varepsilon)_{11}}.dz^1$ is the arclength element of $I \subset (M,g_\varepsilon)$. 
When $\varepsilon \rightarrow 0$, then the limit of (\ref{curvatura integra do pico}) is given by
\begin{equation}
\label{contribuicao gap curvatura total}
-\int_I \frac{q(s)}{2 \sqrt{g_{11}(s) \det g (s)}}.ds. \end{equation}

On the other hand it is straightforward that the geodesic curvature of $I\subset \bar M_+$ with respect to the unit normal vector field 
\[
N = -\frac{g_{12}}{\sqrt{g_{11}}\sqrt{\det g}}\frac{\partial}{\partial x^1} + \frac{ g_{11}}{\sqrt{g_{11}}\sqrt{\det g}} \frac{\partial}{\partial x^2}
\]
pointed towards $\bar M_+$ is given by
\begin{eqnarray}
k_{g+}& = & g\left( \nabla_{ \left( \frac{1}{\sqrt{g_{11}}}\frac{\partial}{\partial x^1} \right) } \left( \frac{1}{\sqrt{g_{11}}}\frac{\partial}{\partial x^1} \right), N \right) \nonumber \\
& = & \frac{\sqrt{\det g}}{2(g_{11})^{3/2}}\left(g^{21} \frac{\partial g_{11+}}{\partial x^1} + 2.g^{22}\frac{\partial g_{12+}}{\partial x^1} - g^{22} \frac{\partial g_{11+}}{\partial x^2} \right). \label{kgmais}
\end{eqnarray}
Analogously the geodesic curvature of $I \subset \bar M_-$ with respect to the unit normal vector field $-N$ pointed towards $\bar M_-$ is given by
\begin{equation}
\label{kgmenos} 
k_{g-} = - \frac{\sqrt{\det g}}{2(g_{11})^{3/2}}\left(g^{21} \frac{\partial g_{11-}}{\partial x^1} + 2.g^{22}\frac{\partial g_{12-}}{\partial x^1} - g^{22} \frac{\partial g_{11-}}{\partial x^2} \right). 
\end{equation}
But
\[
\frac{\partial g_{ij+}}{\partial x^1} = \frac{\partial g_{ij-}}{\partial x^1}
\]
for every $i,j\in \{1,2\}$, what implies that
\begin{eqnarray}
k_{g+} + k_{g-} & = & - \frac{g^{22}\sqrt{\det g}}{2(g_{11})^{3/2}}\left(\frac{\partial g_{11+}}{\partial x^2} - \frac{\partial g_{11-}}{\partial x^2} \right) \nonumber \\
& = & -\frac{1}{2\sqrt{g_{11}\det g}}\left(\frac{\partial g_{11+}}{\partial x^2} - \frac{\partial g_{11-}}{\partial x^2} \right). \label{soma curvatura geodesica}
\end{eqnarray}
Therefore the contribution of the gap (\ref{gap g112}) along $I$ on the total curvature of $(M,g)$ is given by 
\[
\int_I \left( k_{g+} + k_{g-} \right) ds
\]
due to (\ref{contribuicao gap curvatura total}), (\ref{kgmais}), (\ref{kgmenos}) and (\ref{soma curvatura geodesica}).
In summary, piecewise two-dimensional Riemannian manifolds admits nonzero total curvature on subsets of measure zero.

\begin{obs}
\label{o exemplo Finsler}
In this remark we outline how the analysis made for the bidimensional case can be extended for more general cases. 
Let $(M=\mathbb R^n,F)$ be a piecewise smooth Finsler manifold such that $F\vert_{T\bar M_+ \backslash 0}$ and $F\vert_{T\bar M_- \backslash 0}$ are smoothly extendable to the slit tangent bundle of a neighborhood of $\bar M_+$ and $\bar M_-$ respectively.
We consider the canonical coordinates $(x^1, \ldots, x^n)$ on $M$, $(x^1, \ldots, x^n, y^1, \ldots, y^n)$ on $TM$ and $(x^1, \ldots, x^n, y^1, \ldots, y^n,$ $z^1, \ldots, z^n)$ on $\pi^\ast TM$ as in the proof of Theorem \ref{curvatura flag converge uniforme}.
For the horizontal smoothing $(M, F_\varepsilon)$, the flag curvature of the flag $(y,z)$ on $T_xM$ is given by 
\[
K_\varepsilon (y,z):=\frac{z^i(y^j (R_\varepsilon)_{jikl}y^l)z^k}{g_\varepsilon(y,y)g_\varepsilon(z,z)-[g_\varepsilon(y,z)]^2}.
\]
For every $(i,j,k,l)$, it follows from (\ref{termo nao pico gij}) that 
\[
\frac{z^i y^j y^l z^k}{g_\varepsilon(y,y)g_\varepsilon(z,z)-[g_\varepsilon(y,z)]^2}
\]
converges uniformly to
\[
\frac{z^i y^j y^l z^k}{g(y,y)g(z,z)-[g(y,z)]^2}
\]
on $U$ when $\varepsilon$ converges to zero.
In particular, there exist $\varepsilon^\prime >0$ such that
\[
\sup_{\substack{\varepsilon \in (0,\varepsilon^\prime) \\ x \in U \\(y,z) \in \mathcal F_x}} \left\vert \frac{z^i y^j y^l z^k}{g_\varepsilon(y,y)g_\varepsilon(z,z)-[g_\varepsilon(y,z)]^2} \right\vert < \infty.
\]
As in the bidimesional example, $(R_\varepsilon)^{\ i}_{j\ kl}$ is given in terms of
\begin{itemize}
\item $(g_\varepsilon)_{ij}$ and their derivatives of order up to two;
\item $(g_\varepsilon)^{ij}$ and their derivatives of order up to one
\end{itemize}
(see (\ref{Rjikl hh}) and (\ref{Gamma almost metric compatibility})).
The flag curvature can be split in terms that are bounded as $\varepsilon$ converges to zero and terms that are unbounded.
The latter are related to the terms of type
\[
\frac{\partial^2 (g_\varepsilon)_{ij}}{\partial (x^n)^2},
\] 
and we can try to study the geometric meaning of these terms like we did in the two-dimensional piecewise smooth Riemannian case.
\end{obs}

\section{Final remarks}
\label{secao obs finais}

In this section, we make final remarks and we suggest some problems.

The technique applied in Section \ref{secao exemplos} can be probably applied in the study of piecewise smooth Finsler structures, where $F$ behaves locally as in the example.
In particular, it will be interesting to study the influence of the Cartan tensor on the unbounded term of the flag curvature.

We can try to go a little bit further and work with piecewise smooth Finsler structures, but allowing that $F$ isn't defined on a subset of measure zero.
We have geometric objects such as piecewise flat Finsler surfaces as defined in \cite{Xu-Deng}. 
In order to formalize this idea, a natural trial is to consider structures $F:TM \rightarrow \mathbb R$ in some local Sobolev spaces such that $F(x, \cdot)$ is an asymmetric norm.
Notice that mollifier smoothings can make sense even if $F$ isn't defined on a subset of measure zero.

Another important issue to be addressed in the future is to study some criteria to prove that the asymptotic behaviour of a geometric object when $\varepsilon$ converges to zero is independent of the choice of the mollifier smoothing.

\end{document}